\documentclass{llncs}
\usepackage{llncsdoc}
\usepackage{algorithm}
\usepackage[noend]{algpseudocode}
\usepackage{calc}
\usepackage{amsmath}
\usepackage{slashbox}
\usepackage{amssymb}
\usepackage{multirow}
\usepackage{booktabs}
\usepackage{color}
\usepackage{fca}
\usepackage{caption}
\usepackage{extarrows}
\usepackage{orcidlink}
\usepackage{cite}
\usepackage []{hyperref}
\hypersetup{hypertex=true,
colorlinks=true,
linkcolor=blue,
anchorcolor=blue,
citecolor=blue}
\title{Meta-Modelling in Formal Concept Analysis}
\author{Yingjian Wang~\orcidlink{0009-0005-3492-7992}}
\institute{Technische Universit$\ddot{\text{a}}$t Dresden}
\begin{document}
\maketitle

\begin{abstract}Formal Concept Analysis starts from a very basic data structure comprising objects and their attributes. Sometimes, however, it is beneficial to also define attributes of attributes, viz., meta-attributes. In this paper, we use Triadic Formal Concept Analysis, a triadic approach to Formal Concept
Analysis, to develop a framework for this kind of meta-modelling in Formal Concept Analysis, including formal definitions and appropriate visualizations.
\end{abstract}

\section{Introduction}

\textbf{Formal Concept Analysis} (FCA) was introduced as a mathematical theory modelling
the concept of `concepts' in terms of lattice theory. Introduced by Rudolf Wille in 1982\cite{wille2009restructuring}, FCA has since then grown rapidly and has successfully been applied to many fields.

\textbf{Triadic Concept Analysis} was introduced by Lehmann and Wille as a natural extension of (dyadic) formal concept analysis\cite{lehmann1995triadic}. As a triadic approach to Formal Concept Analysis, Triadic Concept Analysis is based on a formalization
of the triadic relation connecting formal objects, attributes and conditions. This triadic approach can be applied for analysing the ternary relations among \emph{objects\emph{,} attributes \emph{and} meta-attributes}.

We present the background knowledge in the next chapter. Then we start with the construction of triadic context from two diadic contexts. Later we will show the isomorphism between diadic concept lattice and triadic concept lattice. At last, we have insight into triadic implications and conditional implications.

\section{Preliminaries}
\subsection{Formal Concept Analysis}

To allow a mathematical
description of extensions and intensions, FCA starts with a (formal) context.\cite{ganter2012book}

\begin{definition}A (formal) context is a triple $\mathbb{K}:=(G, M, I)$, where $G$ is a set whose
elements are called objects, $M$ is a set whose elements are called attributes, and $I$
is a binary relation between $G$ and $M$ (i.e. $I \subseteq G\times M$). $(g,m)\in I$ is read as ``the object g has attribute m''.\end{definition}

Normally a formal context is represented as a cross table: for each object, there is a row uniquely stands for it. And for each attribute, a column is uniquely assigned to it.
The intersection of row $g$ with column $m$ is marked whenever $(g,m) \in I$.\cite{ganter2002formal}

In order to check the common attributes shared by a certain set of objects, or the common objects satisfy a certain set of attributes, an operator is defined.

\begin{definition}For $A\subseteq G$, let

\begin{center}
$A' := \{m \in M$ $|$ $\forall g \in A: (g,m) \in I\}$
\end{center}

and, for $B \subseteq M$, let

\begin{center}
$B' := \{g \in G$ $|$ $\forall m \in B: (g,m) \in I\}.$
\end{center}

$\cdot'$ is the derivation operator for $(G,M,I)$.
\end{definition}

\begin{definition}A (formal) concept of a formal context $(G,M,I)$ is a pair $(A,B)$ with $A \subseteq G$,
$B \subseteq M$, $A' = B$ and $B' = A$. The sets $A$ and $B$ are called the extent and the intent
of the formal concept $(A,B)$, respectively. The set of all formal concepts of $(G,M,I)$ is denoted
$\mathfrak{B}(G,M, I)$.\end{definition}

The description of a concept by extent and intent is redundant. However, this description makes sure that, for a concept $(A, B)$ of $(G, M, I)$, $A$ and $B$ are each maximal (with respect to set inclusion) with the property $A \times B \subseteq I$.\cite{ganter2002formal}

\begin{definition}Let $(A_{1},B_{1})$ and $(A_{2},B_{2})$ be formal concepts of $(G,M, I)$. We say
that $(A_{1},B_{1})$ is a \textbf{subconcept} of $(A_{2},B_{2})$ (and, equivalently, that $(A_{2},B_{2})$ is a
\textbf{superconcept} of $(A_{1},B_{1})$) iff $A_{1} \subseteq A_{2}$. We use the $\leq$-sign to express this relation
and thus have

\begin{center}
$(A_{1},B_{1})\leq (A_{2},B_{2})\Leftrightarrow A_{1} \subseteq A_{2}$.
\end{center}

The set of all formal concepts $\mathfrak{B}(G,M, I)$ of $(G,M, I)$, ordered by this relation, is denoted
\begin{center}
$\underline{\mathfrak{B}}(G,M, I)$
\end{center}

and is called the concept lattice of the formal context $(G,M, I)$. The concept lattice can be represented as a line diagram.
\end{definition}

\subsection{Example}
Here is an example for the above mentioned definitions. Considering the following fragment of a knowledge base:

\emph{Jack is a student.}

\emph{Jason is a librarian.}

\emph{James is a professor.}

...

\emph{A student studies in a classroom.}

\emph{A student studies in a library.}

\emph{A librarian works in an office.}

\emph{A librarian organizes books in a library.}

\emph{A professor teaches in a classroom.}

\emph{A professor works in an office.}

...

This example can be interpreted as a ternary relationship among objects, attributes and meta-attributes, since careers are attributes for individuals and locations are attributes for careers. By definition, we can form two contexts. One for individuals and careers(Table 1), the other one for careers and locations(Table 2). Fig. 1 shows the relevant concept lattice diagrams for the partial contexts. Each circle stands for a formal concept. The names of the attributes are given. Each name is attached to one
of the formal concepts and is written slightly above the respective circle. Dually, the names of objects are written below the respective circles.

The concept order is expressed via the edges that connect different circles. For two circles that are connected by an edge, the lower circle stands for a subconcept of the higher one.
With the help of the edges we can read from the diagram which concepts are subconcepts
of which other concepts, and which objects have which attributes. To do
so, one has to follow ascending paths in the diagram.\cite{ganter2002formal}

For example, consider the object $Student$. From the corresponding circle we can
reach, via ascending paths, the attributes $Classroom$ and $Library$, which means $Student$ has the property $Classroom$ and $Library$. If we consider the attribute \emph{Office}. From the corresponding circle we can reach, via descending paths, the objects $Professor$ and $Librarian$, which means the attribute $Office$ holds for objects $Professor$ and $Librarian$.

\begin{table}[!htb]
\begin{center}\begin{cxt}%

\cxtName{\tiny{\backslashbox{Individuals}{Careers}}}%
\att{Student}%
\att{Librarian}%
\att{Professor}%
\att{\ \ldots}%
\obj{x..p}{Jack}
\obj{.x.p}{Jason}
\obj{..xp}{James}
\obj{qqqo}{$\vdots$}

\end{cxt}

\caption{Context for individuals and careers.}
\label{tab:float}\end{center}

\begin{center}\begin{cxt}%

\cxtName{\tiny{\backslashbox{Careers}{Locations}}}%
\att{Classroom}%
\att{Library}%
\att{Office}%
\att{\ \ldots}%
\obj{xx.p}{Student}
\obj{.xxp}{Librarian}
\obj{x.xp}{Professor}
\obj{qqqo}{$\vdots$}

\end{cxt}

\caption{Context for careers and locations.}
\label{tab:float}\end{center}
\end{table}
\begin{table}[!htb]
\begin{diagram}{40}{45}
\Node{1}{40}{10}
\Node{2}{0}{40}
\Node{3}{40}{40}
\Node{4}{80}{40}
\Node{5}{40}{70}
\Edge{1}{2}
\Edge{1}{3}
\Edge{1}{4}
\Edge{2}{5}
\Edge{3}{5}
\Edge{4}{5}
\leftAttbox{2}{2}{2}{Student}
\rightAttbox{3}{0}{2}{Professor}
\rightAttbox{4}{2}{2}{Librarian}
\leftObjbox{2}{2}{2}{Jack}
\rightObjbox{3}{2}{2}{James}
\rightObjbox{4}{2}{2}{Jason}
\end{diagram}
\begin{diagram}{400}{75}
\Node{1}{220}{0}
\Node{2}{180}{30}
\Node{3}{220}{30}
\Node{4}{260}{30}
\Node{5}{180}{60}
\Node{6}{220}{60}
\Node{7}{260}{60}
\Node{8}{220}{90}
\Edge{1}{2}
\Edge{1}{3}
\Edge{1}{4}
\Edge{2}{5}
\Edge{2}{6}
\Edge{3}{5}
\Edge{4}{6}
\Edge{4}{7}
\Edge{3}{7}
\Edge{5}{8}
\Edge{6}{8}
\Edge{7}{8}
\leftAttbox{5}{2}{2}{Classroom}
\rightAttbox{6}{0}{2}{Library}
\rightAttbox{7}{2}{2}{Office}
\leftObjbox{2}{2}{2}{Student}
\rightObjbox{3}{2}{2}{Professor}
\rightObjbox{4}{2}{2}{Librarian}\label{fig1}
\end{diagram}
\begin{center}
  Fig. 1: Relevant concept lattice diagrams for the partial contexts.
\end{center}
\end{table}

Considering the two contexts for the above mentioned example. The attributes for the first context are the objects for the second context. However, FCA is not sufficient to analyse this ternary relation. Considering individuals as objects, careers as attributes and locations as meta-attributes, the relation between objects and meta-attributes can not be directly analysed via (diadic) FCA.

\section{Meta-Modelling in Formal Concept Analysis}

In this section, we first present the relevant knowledge about Triadic Concept Analysis. Then we specify how to derive the corresponding triadic contexts for representing the ternary relationship.

\subsection{Triadic Formal Concept Analysis}
\begin{definition}A triadic context $\mathbb{K}:=(G, M, B, Y)$ consists of sets $G$, $M$ and $B$(stands for objects, attributes and conditions, respectively), and a ternary relation $Y \subseteq G \times M \times B$. An incidence $(g,m,b)\in Y$ is read as ``the object g has the attribute m under the condition b''.
\end{definition}
A triadic context can be described by a
three-dimensional cross table, under suitable permutations of
rows, columns, and layers of the cross table.\cite{lehmann1995triadic}

\begin{definition}For a triadic context $\mathbb{K}:=(K_{1}, K_{2}, K_{3}, Y)$, $\{i, j, k\} = \{1, 2, 3\}$ with $j < k$ and for $X \subseteq K_{i}$ and
$Z \subseteq  K_{j} \times K_{k}$, the (i)-derivation operators are defined by
\begin{center}
$X^{(i)} := \{(a_{j}, a_{k}) \in K_{j} \times K_{k}$ $|$ $\forall a_{i}\in X, a_{i}\times a_{j}\times a_{k} \subseteq Y\}$,\end{center}\begin{center}
$Z^{(i)} := \{a_{i} \subseteq K_{i}$ $|$ $\forall a_{j}\times a_{k}\in Z\in X, a_{i}\times a_{j}\times a_{k} \subseteq Y\}$.

\end{center}

\end{definition}

 With (i)-derivation operators, one can analyse triadic contexts in the view of diadic contexts. Lehmann and Wille have introduced different derivation operators, including the one defined in Definition 6, for computing the triadic concepts\cite{lehmann1995triadic}.

\begin{definition}A triadic concept is a triple $(X_{1},X_{2},X_{3})$ with $X_{1}\subseteq G$, $X_{2}\subseteq M$, $X_{3}\subseteq B$ and $X_{1}\times  X_{2}\times X_{3}\subseteq Y$, such that none of the three components can be enlarged without violating the condition $X_{1}\times  X_{2}\times  X_{3}\subseteq Y$ .
We call $X_{1}$ the extent, $X_{2}$ the intent and $X_{3}$
the modus of the formal triadic concept. The set of all triadic formal concepts of $(G,M,B, Y)$ is denoted
$\mathfrak{B}(G,M,B,Y)$.

\end{definition}
For example, a triadic concept $(a, b, c)$ is read as \emph{``object a has attribute b under condition c''}. In formal concept analysis, formal concepts are structured by partial order. Compared with Formal concepts, the structure of triadic concepts is given by the
set inclusion in each of the three components of the triadic concepts.\cite{lehmann1995triadic}

\begin{definition}Let $(P_{1},P_{2},P_{3})$ and $(Q_{1},Q_{2},Q_{3})$ be triadic concepts of $(G,M,B,Y)$. We use $\lesssim_{1},\lesssim_{2}$ and $\lesssim_{3}$ to express the quasiorders:
\begin{center}
$(P_{1},P_{2},P_{3})\lesssim_{i} (Q_{1},Q_{2},Q_{3})\Leftrightarrow P_{i} \subseteq Q_{i}$, for $i \in\{ 1, 2, 3\}$.
\end{center}

And we use $\sim_{1},\sim_{2}$ and $\sim_{3}$ to express the corresponding equivalence relations:
\begin{center}
$(P_{1},P_{2},P_{3})\sim_{i} (Q_{1},Q_{2},Q_{3})\Leftrightarrow P_{i} = Q_{i}$, for $i \in\{ 1, 2, 3\}$.
\end{center}

The set of all triadic concepts of $(G,M,B,Y)$, ordered by this relation, is denoted
\begin{center}
$\underline{\mathfrak{B}}(G,M,B,Y):=(\mathfrak{B},\lesssim_{1},\lesssim_{2}, \lesssim_{3})$
\end{center}

and is called the triadic relational structure of the triadic context $(G,M,B,Y)$.\end{definition}

The following sets can be identified as the ordered sets of extents, intents and modus:

\begin{center}
$\underline{\mathfrak{B}_\text{{extent}}}:=(\mathfrak{B}/\sim_{1},\leq_{1})$, $\underline{\mathfrak{B}_\text{{intent}}}:=(\mathfrak{B}/\sim_{2},\leq_{2})$ and $\underline{\mathfrak{B}_{\text{modus}}}:=(\mathfrak{B}/\sim_{3},\leq_{3})$.
\end{center}

In order to draw the corresponding triadic diagram, the relational structure can be understood as a combination of two types of structures:
the geometric structure $(\mathfrak{B}, \sim_{1},\sim_{2}$ , $\sim_{3})$ and the ordered structures $(\mathfrak{B}/\sim_{i},\leq_{i})$ for $i = 1,2,3$.\cite{lehmann1995triadic} The three equivalence relations $\sim_{1},\sim_{2}$ and $\sim_{3}$ are represented by three systems of parallel lines
in the plane and the elements of one equivalence class are located on one line of
the corresponding parallel system. The ordered structures are represented
by line diagrams. An example will be shown later.

\subsection{Meta-Modelling in Triadic Concept Analysis}

\begin{definition}Given formal contexts $\mathbb{K}_{1}:=(G, M, I)$ and $\mathbb{K}_{2}:=(M, B, J)$, their corresponding triadic context is denoted $\mathbb{K}:=(G, M, B, Y)$, where $(g,m)\in I \wedge (m,b)\in J\Longleftrightarrow (g,m,b)\in Y$. 

\end{definition}

To make Definition 9 robust, we need to do some modification on $\mathbb{K}_{1}$ and $\mathbb{K}_{2}$ for the following cases. For the case that some objects, denoted $G_{{\emptyset}}$, have no attribute, we introduce a new irrelevant attribute $E_\text{{m}}$ that holds for all the objects, including $G_{\emptyset}$. Such $E_\text{{m}}$ has all the meta-attributes as well to solve the case that there exists some meta-attributes $B_{\emptyset}$ who holds for no attribute.

Considering the above mentioned example, if we consider the formal contexts shown in Table 1 and Table 2 as $\mathbb{K}_{1}$ and $\mathbb{K}_{2}$, we can construct the corresponding triadic context $\mathbb{K}$ by Definition 9. The corresponding triadic context and its triadic diagram are shown in Table 3 and Fig. 2. For simplicity, we use \emph{a, b, c} instead of \emph{Jack, Jason \emph{and} James}; \emph{1, 2, 3} instead of \emph{Student, Librarian \emph{and} Professor}; $\alpha, \beta, \gamma$ instead of \emph{Classroom, Library \emph{and }Office}.

\begin{table}[!htb]
  \centering
  \begin{tabular}{|c||c|c|c|c|c|c|c|c|c|}

\hline

Meta-attributes& \multicolumn{2}{r}{$\alpha$} &

& \multicolumn{2}{r}{$\beta$}&&

\multicolumn{2}{r}{$\gamma$}&\\
\hline

\tiny{\backslashbox{Objects}{Attributes}}& 1 & 2 & 3 & 1 & 2 & 3 & 1 & 2 & 3    \\
 \hline
 \hline
 a&$\times$&\color{white}{w} & \color{white}{w}&$\times$&\color{white}{w} &\color{white}{w} & \color{white}{w}& \color{white}{w} &\\
 \hline
 b&&&&&$\times$&&&$\times$&\\
 \hline
 c&&&$\times$&&&&&&$\times$\\
 \hline

\end{tabular}
\caption{Corresponding triadic context for the previous example.}

\end{table}

\begin{figure}[!htb]
  \centering
  \includegraphics[width=2in]{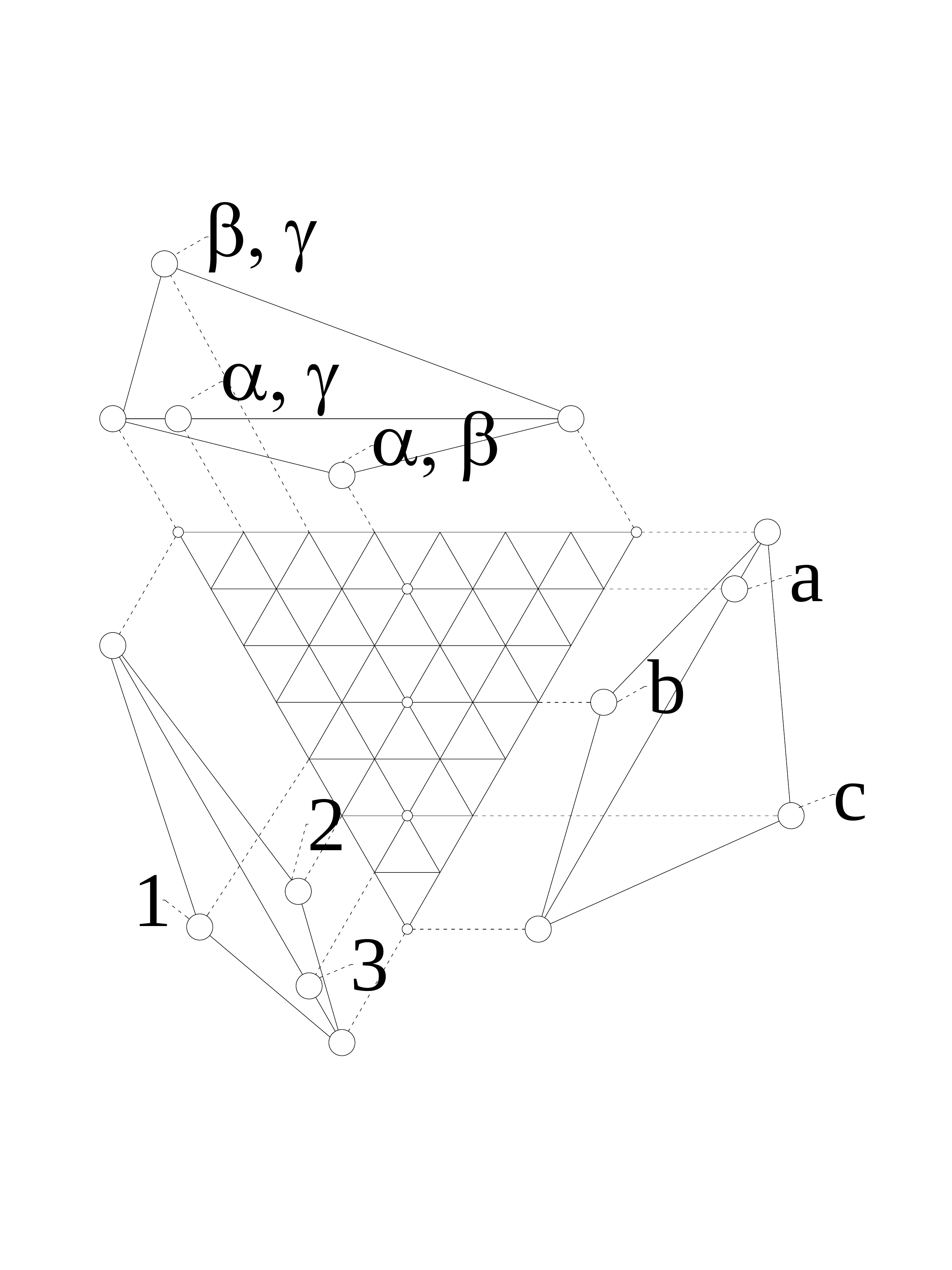}\\
  \begin{center}
  Fig. 2: Triadic diagram for $\mathbb{K}.$
\end{center}
\end{figure}

The geometric structure of the triadic concepts is represented by the triangular
pattern in the center of the diagram. The circles represent the triadic concepts. The horizontal lines stand for the equivalence classes of $\sim_{1}$, the
lines ascending to the right stand for the equivalence classes of $\sim_{2}$, and the lines ascending to the left stand for the equivalence classes of $\sim_{3}$. The perforated lines indicate the connection to the extent diagram on the
right, to the intent diagram on the lower left, and to the modus diagram above.
A circle of the line diagram on the right represents the extent consisting of those objects whose signs are attached to this circle or a circle below. The intents and
modi can analogously be read from the diagram where the intents get larger
from the upper left to the lower right and the modi get larger from the upper
right to to the lower left.\cite{lehmann1995triadic}\cite{Biedermann:1997:TDR:645490.657192}

For instance, in Fig. 2, the next circle
vertically above the lowest circle connects horizontally with the extent $\{c\}$, to
the lower left with the intent $\{3\}$, and to the upper left with the modus
$\{\alpha, \gamma\}$; hence it represents the triadic concept $({c}, {3}, \{\alpha, \gamma\})$.

Recall that $\mathbb{K}$ is constructed from $\mathbb{K}_{1}$ and $\mathbb{K}_{2}$. We find that the concept lattice diagram for $\mathbb{K}_{1}$ is isomorphic to the extent diagram for $\mathbb{K}$ in our example. However, we cannot conclude that the concept lattice diagrams of $\mathbb{K}_{1}$ and $\mathbb{K}_{2}$ are bijective to some line diagrams of $\mathbb{K}$ since the concept lattice diagram of $\mathbb{K}_{2}$ is an counter example. In the next chapter we will modify the triadic context to satisfy this `isomorphism'.

\section{Isomorphism}
We already defined the way to construct the triadic context $\mathbb{K}$ from diadic contexts $\mathbb{K}_{1}$ and $\mathbb{K}_{2}$. We want to show that after the modification on $\mathbb{K}_{1}$ and $\mathbb{K}_{2}$, the concept lattice diagram is isomorphic to the line diagrams of $\mathbb{K}$. Moreover, we can derive all the formal concepts of $\mathbb{K}_{1}$ and $\mathbb{K}_{2}$ from the triadic diagram of $\mathbb{K}$ in this chapter.

\begin{definition}Two ordered sets $\mathbb{V}$ and $\mathbb{W}$ are isomorphic $(\mathbb{V} \cong \mathbb{W} )$, if there
exists a bijective mapping $\varphi: \mathbb{V} \rightarrow \mathbb{W}$ with $x \leq y \Leftrightarrow \varphi(x) \leq \varphi(y)$ for $x, y\in \mathbb{V}$. The mapping
$\varphi$ is then called order isomorphism between $\mathbb{V}$ and $\mathbb{W}$ .\end{definition}

Before we have insight into the isomorphism of the diagrams, two lemmas are proved as appetizers, as they make an effort to the proving of the isomorphism in the later part of this chapter.

\begin{lemma}Given a formal context $\mathbb{C}:=(G, M, I)$, its intent set and extent set denoted $INT$ and $EXT$ respectively. After deleting the column of a random attribute $m\in M$, we get a new formal context $\mathbb{C}_{*}:=(G, M/m, I_{*})$, where $I_{*}:=I/I_{m}$ and $I_{m}:=\{(g, m)\in I$ $|$ $g\in G\}$. Its intent set and extent set denoted $INT_{*}$ and $EXT_{*}$ respectively. Then $INT_{*}$ is equal to $INT/m$ and $EXT_{*}\subseteq EXT$. $INT/m$ stands for the set constructed from $INT$ by removing all the occurrences of $m$ in each element of $INT$.\end{lemma}


\begin{proof}$\mathbb{C}$ and $\mathbb{C}_{*}$ share the same set of objects $G$. 
For every subset $O$ of $G$, $(O'', O')$ is a concept of $\mathbb{C}$ and $(O^{\diamond\diamond}, O^{\diamond})$ is a concept of $\mathbb{C}_{*}$. To avoid
ambiguity, we use $\cdot^{\diamond}$ as the derivation operator for $\mathbb{C}_{*}$, instead of $\cdot'$. If $O^{\diamond}$ is equal to $O'/m$, then follows Lemma 1. We now show that $O^{\diamond}=O'/m$ and $O^{\diamond\diamond}\in EXT$ by discussing different cases.





1. Assume $m \notin O'$, then $(O'', O')$ is a concept of $\mathbb{C}$ and $(O^{\diamond\diamond}, O^{\diamond})$ is a concept of $\mathbb{C}_{*}$ while $O''$ is equal to $O^{\diamond\diamond}$ and $O'$ is equal to $O^{\diamond}$. As $m \notin O'$ by assumption, $O^{\diamond}$ is equal to $O'/m$ and $O^{\diamond\diamond}=O''\in EXT$. $(O^{\diamond\diamond}, O^{\diamond})$ is a concept of $\mathbb{C}$ as well.




2. Assume $m\in O'$, then $(O'', O')$ is a concept of $\mathbb{C}$ and $(O^{\diamond\diamond}, O^{\diamond})$ is a concept of $\mathbb{C}_{*}$. As $m \in O'$ by assumption, $O^{\diamond}$ is equal to $O'/m$ and $O''\subseteq O^{\diamond\diamond}$ by Definition 2. Moreover, if $O^{\diamond}\in INT$, $(O^{\diamond\diamond}, O^{\diamond})$ is a concept of $\mathbb{C}$. Otherwise, $(O^{\diamond\diamond}, O^{\diamond}\cup m)$ is a concept of $\mathbb{C}$. $O^{\diamond\diamond}\in EXT$ anyway.

In both cases, $O^{\diamond}$ is equal to $O'/m$ and $O^{\diamond\diamond}\in EXT$. In another word, $INT_{*}$ is equal to $INT/m$ and $EXT_{*}\subseteq EXT$.\qed\end{proof}







Now let's think about a variant of Lemma 1. What if we are talking about $\mathbb{C}_{\star}:=(G, M, I_{\star})$ rather then $\mathbb{C}_{*}:=(G, M/m, I_{*})$, where $I_{\star}:=I/I_{m}$ as well? The difference is that, instead of deleting the column $m$, we only remove the relations $I_{m}$, which means that $m'=\emptyset$ holds for $\mathbb{C}_{\star}$ while it doesn't hold for $\mathbb{C}_{*}$. Can we apply Lemma 1 on $\mathbb{C}_{\star}$ as well directly? Unfortunately, the answer is no.

\begin{table}[!htb]

\begin{center}
\begin{cxt}%
\cxtName{$\mathbb{C}$}%
\att{$1$}%
\att{$2$}%
\att{$3$}%
\obj{xxx}{a}
\obj{.xx}{b}
\obj{..x}{c}

\end{cxt}
\begin{cxt}%
\centering
\cxtName{$\mathbb{C}_{\star}$}%
\att{1}%
\att{2}%
\att{3}%
\obj{xx.}{$\color{white}{.}$a}
\obj{.x.}{$\color{white}{.}$b}
\obj{...}{$\color{white}{.}$c}

\end{cxt}

\end{center}
\caption{$\mathbb{C}$ and $\mathbb{C}_{\star}$.}

\end{table}

Consider the example given in Table 4, there are three concepts for $\mathbb{C}$, $(\{a\}, \{1, 2, 3\})$, $(\{a, b\}, \{ 2, 3\})$ and $(\{a, b, c\}, \{3\})$. $\mathbb{C}_{\star}$, who is constructed from $\mathbb{C}$ by removing the occurrences of attribute \emph{3}, has four concepts, $(\{a\}, \{1, 2\})$, $(\{a, b\}, \{ 2\})$, $(\{a, b, c\}, \emptyset)$ and $(\emptyset, \{1,2,3\})$. It is easily seen that neither ``$INT_{\star}$ is equal to $INT/m$'' nor ``$EXT_{\star}\subseteq EXT$''. However, with more constraints, we can still find out the relationships between $INT$ and $INT_{\star}$, $EXT$ and $EXT_{\star}$.

\begin{lemma}Given a formal context $\mathbb{C}:=(G, M, I)$, its intent set and extent set denoted $INT$ and $EXT$ respectively. After removing the occurrences of a random attribute $m$ from $M$, we get a new formal context $\mathbb{C}_{\star}:=(G, M, I_{\star})$, where $I_{\star}:=I/I_{m}$ and $I_{m}:=\{(g, m)\in I$ $|$ $g\in G\}$. Its intent set and extent set denoted $INT_{\star}$ and $EXT_{\star}$ respectively. Then the following statements hold.

1. If $(\emptyset, M)$ is a concept of $\mathbb{C}$, $INT_{\star}$ is equal to $((INT-{M})/m)\cup {M}$ and $EXT_{\star}\subseteq EXT$.

2. Otherwise $INT_{\star}$ is equal to $(INT/m)\cup {M}$ and $(EXT_{\star}-\emptyset )\subseteq EXT$.

$(INT-{M})$ means that if $M\in INT$, remove ${M}$ from $INT$. $INT/m$ stands for the set constructed from $INT$ by removing all the occurrences of $m$ in each element of $INT$. $(EXT_{\star}-\emptyset)$ means that if $\emptyset\in EXT_{\star}$, remove $\emptyset$ from $EXT_{\star}$.\end{lemma}

\begin{proof}The proof is based on the knowledge of Lemma 1.

We already know that $\mathbb{C}_{*}:=(G, M/m, I_{*})$, where $I_{*}:=I/I_{m}$ and $I_{m}:=\{(g, m)\in I$ $|$ $g\in G\}$. The only difference between  $\mathbb{C}_{*}$ and $\mathbb{C}_{\star}$ is that $\mathbb{C}_{\star}$ has attribute $m$, who holds for no object. If we find out the relation between $INT_{*}$ and $INT_{\star}$, we can use Lemma 1 to find out the relation between $INT$ and $INT_{\star}$ as well. Assume that $(A, B)$ is a concept of $\mathbb{C}_{*}$, then follows $A\times B\subseteq I_{\star}$ as well.

1. For the case that $(\emptyset, M/m)$ is a concept of $\mathbb{C}_{*}$.

It is easy to see that in this case, $(\emptyset, M)$ is a concept of $\mathbb{C}$, since $(M/m)'=\emptyset$ in $\mathbb{C}$, which implies $M'=\emptyset$. Since $\emptyset \times m\in I_{\star}$, $(\emptyset, M)$ is a concept of $\mathbb{C}_{\star}$. For any other concept of $\mathbb{C}_{*}$, for example $(A, B)$, since $A$ and $B$ are each maximal (with respect to set inclusion) with the property $A\times B\subseteq I$, $(A, B)$ is a concept of $\mathbb{C}_{\star}$ as well. Now we get a set of concepts for $\mathbb{C}_{\star}$, it can be expressed as $\mathfrak{B}_{\star partial}:=(\mathfrak{B}_{*}-(\emptyset, M/m))\cup (\emptyset, M)$. In truth, it is exactly $\mathfrak{B}_{\star}$. Assume that $\mathfrak{B}_{\star}$ is not equal to $\mathfrak{B}_{\star partial}$, then there exist a concept $(A, B)$ such that $(A, B)\in \mathfrak{B}_{\star}$ and $(A, B)\notin \mathfrak{B}_{\star partial}$. Since $(A, B)$ is not $(\emptyset, M)$, which means that $m\notin B$, and $A$ and $B$ are each maximum with respect to $A\times B\in I_{\star}$, $A$ and $B$ are each maximum with respect to $A\times B\in I_{*}$ too. Then $(A, B)$ is in $\mathfrak{B}_{\star partial}$, conflict occurs. Therefore, $\mathfrak{B}_{\star}:=(\mathfrak{B}_{*}-(\emptyset, M/m))\cup (\emptyset, M)$, which means that $EXT_{\star}$ is equal to $EXT_{*}$ and $INT_{\star}$ is equal to $(INT_{*}-{M/m})\cup M$. By Lemma 1, $EXT_{*}\subseteq EXT$ and $INT_{*}=INT/m$. Therefore $EXT_{\star}\subseteq EXT$ and $\mathbb{C}$, $INT_{\star}=(INT/m-{M/m})\cup {M}=((INT-{M})/m)\cup {M}$.

2. For the case that $(\emptyset, M/m)$ is not a concept of $\mathbb{C}_{*}$.

Every concept of $\mathbb{C}_{*}$, for example $(A, B)$, is a concept of $\mathbb{C}_{\star}$ since $A$ and $B$ are each maximal (with respect to set inclusion) with the property $A\times B\subseteq I$. On the other hand, $m$ is an attribute who holds for no object of $\mathbb{C}_{\star}$, which implies that $(\emptyset, M)$ is a concept of $\mathbb{C}_{\star}$ as well. Now we get $\mathfrak{B}_{\star partial}:=\mathfrak{B}_{*}\cup (\emptyset, M)$. Assume that $\mathfrak{B}_{\star}$ is not equal to $\mathfrak{B}_{\star partial}$, then there exist a concept $(A, B)$ such that $(A, B)\in \mathfrak{B}_{\star}$ and $(A, B)\notin \mathfrak{B}_{\star partial}$. Since $(A, B)$ is not $(\emptyset, M)$, which means that $m\notin B$, and $A$ and $B$ are each maximum with respect to $A\times B\in I_{\star}$, $A$ and $B$ are each maximum with respect to $A\times B\in I_{*}$ too. Then $(A, B)$ is in $\mathfrak{B}_{\star partial}$, conflict occurs. Therefore, $\mathfrak{B}_{\star}:=\mathfrak{B}_{*}\cup (\emptyset, M)$, which means that $EXT_{\star}$ is equal to $EXT_{*}$ and $INT_{\star}$ is equal to $INT_{*}\cup M$.

If $(\emptyset, M)$ is a concept of $\mathbb{C}$, since $(\emptyset, M/m)$ is not a concept of $\mathbb{C}_{*}$, which means there exist an object who has all the attributes $M/m$, $m$ is the attribute who hold for no object in $\mathbb{C}$. Therefore $(\emptyset, M)$ is a concept of $\mathbb{C}$. Since $M$ is the only intent of $\mathbb{C}$ that $m$ is included and $M/m$ is an intent of $\mathbb{C}$ as well, we can conclude that $INT_{*}=INT/m=(INT-M)/m$. It is clear to see that in this case $\mathbb{C}$ and $\mathbb{C}_{\star}$ are equal, then it is sure that $EXT_{\star}=EXT$ and $INT_{\star}=INT$. Therefore $EXT_{\star}\subseteq EXT$ and $INT_{\star}=INT=INT_{*}\cup M=((INT-M)/m)\cup M$.

If $(\emptyset, M)$ is not a concept of $\mathbb{C}$, according to Lemma 1, $EXT_{*}\subseteq EXT$ and $INT_{*}=INT/m$. Therefore $EXT_{\star}\subseteq EXT$ and $INT_{\star}=INT=INT_{*}\cup M=(INT/m)\cup M$.\qed\end{proof}





Now we find out the relationships between $INT$ and $INT_{\star}$, $EXT$ and $EXT_{\star}$. Moreover, we can generalize Lemma 1 and Lemma 2 to the case that a set $M_\text{remove}\subseteq M$.

\begin{lemma}Given a formal context $\mathbb{C}:=(G, M, I)$, its intent set and extent set denoted $INT$ and $EXT$ respectively. After deleting the columns of several random attributes $M_\text{remove}$ from $M$, we get a new formal context $\mathbb{C}_{*}:=(G, M/M_\text{remove}, I_{*})$, where $I_{*}:=I/I_{M_\text{remove}}$ and $I_{M_\text{remove}}:=\{(g, m)\in I$ $|$ $g\in G$, $m\in M_\text{remove}\}$. Its intent set and extent set denoted $INT_{*}$ and $EXT_{*}$ respectively. Then $INT_{*}$ is equal to $INT/M_\text{remove}$ and $EXT_{*}\subseteq EXT$. $INT/M_\text{remove}$ stands for the set constructed from $INT$ by removing all the occurrences of $M_\text{remove}$ in each element of $INT$.\end{lemma}

The proof is trivial once we apply Lemma 1 recursively. Benefits from the proof above, the set of formal concepts $\mathfrak{B}_{*}$ can be constructed via Algorithm 1. For the case that we require only $INT_{*}$, rather than $\mathfrak{B}_{*}$, as the output, Algorithm 1 may require $INT$ rather than $\mathfrak{B}$ as input.

\begin{algorithm}
\caption{Generating $\mathfrak{B}_{*}$ from $\mathfrak{B}$ (for Lemma 3)}
 Input: $\mathbb{C}, M_\text{remove}$ and $\mathfrak{B}$.

 Output: $\mathfrak{B}_{*}$ and $\mathbb{C}_{*}$.

 For every formal concept $(ext, int)$ of $\mathbb{C}:=(G, M, I)$, do the following:
  \begin{algorithmic}[1]

  \If{$M_\text{remove}\cap int = \emptyset$}
  \State $\mathfrak{B}_{*}\leftarrow \mathfrak{B}_{*}\cup  (ext, int)$
  \Else
  \If{$int/M_\text{remove}\in INT$}
  \State $\mathfrak{B}_{*}\leftarrow \mathfrak{B}_{*}\cup ((int/M_\text{remove})', int/M_\text{remove})$
  \Else \State $\mathfrak{B}_{*}\leftarrow  \mathfrak{B}_{*}\cup (ext, int/M_\text{remove})$
  \EndIf
   \EndIf
\State $\mathbb{C}_{*}\leftarrow(G, M/M_\text{remove}, I/I_{M_\text{remove}})$

  \end{algorithmic}
\end{algorithm}

\begin{lemma}Given a formal context $\mathbb{C}:=(G, M, I)$, its intent set and extent set denoted $INT$ and $EXT$ respectively. After removing the occurrences of several random attribute $M_\text{remove}$ from $M$, we get a new formal context $\mathbb{C}_{\star}:=(G, M, I_{\star})$, where $I_{\star}:=I/I_{M_\text{remove}}$ and $I_{m}:=\{(g, m)\in I$ $|$ $g\in G, m\in M_\text{remove}\}$. Its intent set and extent set denoted $INT_{\star}$ and $EXT_{\star}$ respectively. Then the following statements hold.

1. If $(\emptyset, M)$ is a concept of $\mathbb{C}$, $INT_{\star}$ is equal to $((INT-{M})/M_\text{remove})\cup {M}$ and $EXT_{\star}\subseteq EXT$.

2. Otherwise $INT_{\star}$ is equal to $(INT/M_\text{remove})\cup {M}$ and $(EXT_{\star}-\emptyset )\subseteq EXT$.

$(INT-{M})$ means that if $M\in INT$, remove ${M}$ from $INT$. $INT/m$ stands for the set constructed from $INT$ by removing all the occurrences of $m$ in each element of $INT$. $(EXT_{\star}-\emptyset)$ means that if $\emptyset\in EXT_{\star}$, remove $\emptyset$ from $EXT_{\star}$.\end{lemma}

Similarly, the proof is trivial once we apply Lemma 2 recursively. Benefit from Lemma 4, the set $INT_{\star}$ can be constructed via Algorithm 2. For the case that we already know that if $(\emptyset, M) \subseteq \mathfrak{B}$, Algorithm 2 may require $INT$ rather than $\mathfrak{B}$ as input.

\begin{algorithm}
\caption{Generating $INT_{\star}$ from $INT$ (for Lemma 4)}
 Input: $\mathbb{C}, M_\text{remove}$ and $\mathfrak{B}$.

 Output: $INT_{\star}$ and $\mathbb{C}_{\star}$.
  \begin{algorithmic}[1]

  \If {$(\emptyset, M) \in \mathfrak{B}$}
  \State $INT_{\star} = ((INT-{M})/M_\text{remove})\cup {M}$
  \Else
  \State $INT_{\star} = (INT/M_\text{remove})\cup {M}$
  \EndIf
\State $\mathbb{C}_{\star}\leftarrow(G, M, I/I_{M_\text{remove}})$
  \end{algorithmic}
\end{algorithm}

\begin{lemma}Given formal contexts $\mathbb{K}_{1}:=(G, M, I)$ and $\mathbb{K}_{2}:=(M, B, J)$ and their corresponding triadic context $\mathbb{K}:=(G, M, B, Y)$. If there exists $b\in B$ such that $M\times  \{b\}\subseteq J$ and $(\emptyset, M)$ is a concept of $\mathbb{K}_{1}$,
the concept lattice of $\mathbb{K}_{1}$ is isomorphic to the ordered set of extents of $\mathbb{K}$.\end{lemma}

\begin{proof}Let $\underline{\mathfrak{B}_{1}}$ be the concept lattice of $\mathbb{K}_{1}$ and let $\underline{\mathfrak{B}_\text{{extent}}}$ be the ordered set of extents of $\mathbb{K}$.

Recall Definition 4, $\underline{\mathfrak{B}_{1}}:=(\mathfrak{B}_{1},\leq)$.

Recall Definition 8, $\underline{\mathfrak{B}_\text{{extent}}}:=(\mathfrak{B}/\sim_{1},\leq_{1})$.

For $\underline{\mathfrak{B}_{1}}$ and $\underline{\mathfrak{B}_\text{{extent}}}$, both of them are isomorphic to the set of extents ordered by ``$\subseteq$''. Then proving the isomorphism (between $\underline{\mathfrak{B}_{1}}$ and $\underline{\mathfrak{B}_\text{{extent}}}$) can be interpreted as proving the following proposition:

\emph{For a set $A \subseteq G$, $A$ is an extent of the formal context $\mathbb{K}_{1}$ iff $A$ is an extent of the triadic formal context $\mathbb{K}$.
 }

In another word, \emph{ the extent set of formal context $\mathbb{K}_{1}$ is equal to the extent set of the triadic formal context $\mathbb{K}$.}

This interpretation is stricter than the definition of isomorphism. By proving this statement, we can show that we can find all the diadic formal concepts via triadic concept diagram.

Given formal contexts $\mathbb{K}_{1}:=(G, M, I)$ and $\mathbb{K}_{2}:=(M, B, J)$ and their corresponding triadic context $\mathbb{K}:=(G, M, B, Y)$. By assumption, there exists $b \in B$ such that $(m,b)\in J$ for all $m\in M$. Let $EXT_{1}$ denote the extent set of $\mathbb{K}_{1}$ and $EXT$ the extent set of $\mathbb{K}$. We now show $EXT_{1} = EXT$.

1. $E \in EXT_{1} \Rightarrow E \in EXT$.

We use induction over $|B|$ to prove it.

\noindent \textbf{I.B.} Given formal contexts $\mathbb{K}_{1}:=(G, M, I)$ and $\mathbb{K}_{2}:=(M, \{b\}, J)$ and their corresponding triadic context $\mathbb{K}:=(G, M, \{b\}, Y)$. By assumption, $(m,b)\in J$ for all $m\in M$. Let $EXT_{1}$ denote the extent set of $\mathbb{K}_{1}$ and $EXT$ the extent set of $\mathbb{K}$. We now show $E \in EXT_{1} \Rightarrow E \in EXT$.

Assume $E \in EXT_{1}$, then $(E, E')$ is a concept of $\mathbb{K}_{1}$ and $(E, E', \{b\})$ is a tri-concept of $\mathbb{K}$. Therefore $E \in EXT$.

\noindent \textbf{I.H.} Given formal contexts $\mathbb{K}_{1}:=(G, M, I)$ and $\mathbb{K}_{2}:=(M, B, J)$ with $|B|=n>0$. Their corresponding triadic context denoted $\mathbb{K}:=(G, M, B, Y)$. Let $EXT_{1}$ denote the extent set of $\mathbb{K}_{1}$ and $EXT$ the extent set of $\mathbb{K}$. If there exists $b\in B$ such that $(m,b)\in J$ for all $m \in M$, $E \in EXT_{1} \Rightarrow E \in EXT$.

\noindent \textbf{I.C.} Given formal contexts $\mathbb{K}_{1}:=(G, M, I)$ and $\mathbb{K}_{2}:=(M, B, J)$ with $|B|=n+1>1$($B=\{b_{1}, b_{2},\ldots,b_{n+1}\}$). Their corresponding triadic context denoted $\mathbb{K}:=(G, M, B, Y)$. Let $EXT_{1}$ denote the extent set of $\mathbb{K}_{1}$ and $EXT$ the extent set of $\mathbb{K}$. If there exists $b\in B$ such that $(m,b)\in J$ for all $m \in M$, $E \in EXT_{1} \Rightarrow E \in EXT$.

\noindent \textbf{I.S.} Without loss of generality, assume that such $b \in \{b_{1}, b_{2},\ldots , b_{n}\}$. $M_{b_{n+1}}\subseteq M$ denotes the set of attributes satisfying $b_{n+1}$. Given formal contexts $\mathbb{K}_{1}:=(G, M, I)$ and $\mathbb{K}_{2*}:=(M, B/b_{n+1}, J_{*})$, where $J_{*}:=J/ J_{b_{n+1}}$ and $J_{b_{n+1}}:=\{(m_{*}, b_{n+1})$ $|$ for all $m_{*}\in M_{b_{n+1}}\}$. Their corresponding triadic context denoted $\mathbb{K}_{*}:=(G, M, B/b_{n+1}, Y_{*})$, where $Y_{*}:=Y/ Y_{b_{n+1}}$ and $Y_{b_{n+1}}:=\{(g,m_{*}, b_{n+1})$ $|$ for all $m_{*}\in M_{b_{n+1}}$, $(g,m_{*})\in I$\}. $EXT_{*}$ denotes the extent set of $\mathbb{K}_{*}$. By the hypothesis, $E \in EXT_{1} \Rightarrow E \in EXT_{*}$. We now show $E \in EXT_{1} \Rightarrow E \in EXT$.

Assume $E \in EXT_{1}$, then $(E, E')$ is a concept of $\mathbb{K}_{1}$. By the hypothesis, $(E, D, F)$ is a tri-concept of $\mathbb{K}_{*}$. If $(E, D, b_{n+1})\in J$, $(E, D, F \cup b_{n+1})$ is a tri-concept of $\mathbb{K}$. Therefore $E \in EXT$. Otherwise, $(E, D, F)$ is a tri-concept of $\mathbb{K}$. Therefore $E \in EXT$.

2. $E \in EXT \Rightarrow E \in EXT_{1}$.

Assume $E \in EXT$, then there exists a tri-concept $(E, D, F)$ of $\mathbb{K}$. $(E, D)$ can be regarded as a concept of the context under conditions $F$, name it $\mathbb{K}_\text{F}$. As $\mathbb{K}_\text{F}$ can be constructed from $\mathbb{K}_{1}$ by removing the occurrences of the attributes that do not have the conditions $F$, Lemma 4 can be applied on $\mathbb{K}_{1}$ and $\mathbb{K}_\text{F}$. Since $(\emptyset, M)$ is a concept of $\mathbb{K}_{1}$, $E\in EXT_{F}\subseteq EXT_{1}$. Then holds $E \in EXT \Rightarrow E \in EXT_{1}$.\qed\end{proof}

\begin{lemma}Given formal contexts $\mathbb{K}_{1}:=(G, M, I)$ and $\mathbb{K}_{2}:=(M, B, J)$ and their corresponding triadic context $\mathbb{K}:=(G, M, B, Y)$. If there exists $g\in G$ such that $\{g\}\times M\subseteq I$ and $(M, \emptyset)$ is a concept of $\mathbb{K}_{2}$,
the concept lattice of $\mathbb{K}_{2}$ is isomorphic to the ordered set of modi of $\mathbb{K}$.\end{lemma}

The prove is trivial. We only need to consider the dual contexts $\mathbb{K}_{1}^{-1}:=(M, G, I^{-1})$, $\mathbb{K}_{2}^{-1}:=(B, M, J^{-1})$ and $\mathbb{K}^{-1}:=(B, M, G, Y^{-1})$, then the problem is interpreted to the previous one, which we already proved.

As a conclusion, given formal contexts $\mathbb{K}_{1}:=(G, M, I)$ and $\mathbb{K}_{2}:=(M, B, J)$ and their corresponding triadic context $\mathbb{K}:=(G, M, B, Y)$, after adding a new meta-attribute who holds for all attributes\footnote{For $\mathbb{K}_{2}:=(M, B, J)$, $M$ is the set of attributes, $B$ is the set of meta-attributes.} for $\mathbb{K}_{2}$ and an attribute, with which $(\emptyset, M)$ is for sure a concept of $\mathbb{K}_{1}$, who holds for no object for $\mathbb{K}_{1}$, $EXT \equiv EXT_{1}$ holds. Each circle on the extent diagram stands for an extent $G_{0}$ of $\mathfrak{B}_{1}$. For each node, we search on its corresponding line (which stands for its equivalence class) on the triangular pattern. For each circle\footnote{The circle stands for $(G, M, \emptyset)$, if exists, should not be concerned, since it follows from the definition of triadic concept rather then the relations in $\mathbb{K}_{1}$.} on this line, we find its corresponding attributes, the set of all the relevant attributes $M_{0}$ is the intent of $G_{0}$. This is how we find all the formal concepts $\mathfrak{B}_{1}$.

Dually, after adding a new object who has all the attributes for $\mathbb{K}_{1}$ and a meta-attribute that holds for no attribute for $\mathbb{K}_{2}$, we can find all the formal concepts of $\mathbb{K}_{2}$ via the modus diagram of $\mathbb{K}$. Moreover, for a set $B_{0} \subseteq B$, $B_{0}$ is an intent of the formal context $\mathbb{K}_{2}$ iff $B_{0}$ is a modus of the triadic formal context $\mathbb{K}$.

One should notice that the preconditions for Lemma 5 and Lemma 6 conflict themselves, which means we can not always maintain the isomorphism for both $\mathbb{K}_{1}$ and $\mathbb{K}_{2}$.
\section{Implications}

In this chapter, we start with the definition of \emph{implications} in FCA. Then we introduce \emph{triadic implications} and \emph{conditional attribute implications} in triadic formal contexts and use  \emph{Next Closure algorithm} to compute triadic implications and conditional attribute implications.

\subsection{Implications in FCA}

For the cases that attributes are given while objects are not known, or there are too many attributes to be handled completely. We then have to study the possibly ``representative'' attribute
combinations, which is the very reason that the definition of implication (between attributes) is introduced.

\begin{definition}An implication $X \rightarrow Y$ holds in a context, if every object that has all attributes from X also has all attributes from Y.\end{definition}

\begin{definition}A set $\mathcal{L}$ of implications of a context $(G,M, I)$ is called complete, if
every implication that holds in $(G,M, I)$ follows from $\mathcal{L}$.
A set $\mathcal{L}$ of implications is called non-redundant if no implication in $\mathcal{L}$
follows from other implications in $\mathcal{L}$.\end{definition}

From the definition of implication, we can infer that an implication $X \rightarrow Y$ holds in a context, iff $Y\subseteq X''$$(X'\subseteq Y')$. It is convenient if we can compute a complete and non-redundant set of implications.

\begin{definition}An implication $A \rightarrow B$ follows (semantically) from a set $\mathcal{L}$ of
implications in $M$ if each subset of $M$ respecting $\mathcal{L}$ also respects $A \rightarrow B$.
A family of implications is called closed if every implication following from
$\mathcal{L}$ is already contained in $\mathcal{L}$.\end{definition}

\begin{proposition}If $\mathcal{L}$ is a set of implications in $M$, then
\begin{center}
\rm{Mod}$(\mathcal{L})$ := $\{T \subseteq M$ $|$ $T$ $respects$ $\mathcal{L}\}$
\end{center}
is a closure system on $M$. If $\mathcal{L}$ is the set of all implications of a context, then \rm{Mod}$(\mathcal{L})$
is the system of all concept intents.\end{proposition}

The respective closure operator
\begin{center}
$X \mapsto \mathcal{L}(X)$
\end{center}
can be described as follows: For a set $X \subseteq M$, let
\begin{center}
$X^\mathcal{L} := X \cup \bigcup \{B$ $|$ $A \rightarrow B \in \mathcal{L},$ $A \subseteq X\}$.
\end{center}
Form the sets $X^\mathcal{L}$, $X^{\mathcal{L}\mathcal{L}}$, $X^{\mathcal{L}\mathcal{L}\mathcal{L}}$, . . . until a set $\mathcal{L}(X)$ := $X^{\mathcal{L}...\mathcal{L}}$ is obtained with
$\mathcal{L}(X)^\mathcal{L}$ = $\mathcal{L}(X)$. $\mathcal{L}(X)$ is then the
closure of $X$ for the closure system \rm{Mod}$(\mathcal{L})$.

\begin{definition}$P\subseteq M$ is called a pseudo intent(pseudo closed) of $(G, M, I)$, if $P \neq P''$, and if $Q\subsetneq P$ is a pseudo intent, then $Q''\subseteq P$.\end{definition}

\begin{theorem}The set of implications $\mathcal{L}:=\{P\rightarrow$ $P''$ $|$ $P$ is a pseudo intent$\}$ is non-redundant and complete. We call $\mathcal{L}$ the stem base.\end{theorem}

The proof of Theorem 1 is omitted here, curious readers may check the relevant papers for details.\cite{ganter2012book}\cite{ganter2002formal}

\subsection{Triadic Implications in Triadic Formal Contexts}

There are different ways to define implications in a triadic formal context. In this section, triadic implications, which is mentioned in the dissertation of Ganter and Obiedkov\cite{ganter2004implications}, will be discussed. 

Before we look into the details, the definition of conditional context is specified since it essential for later discussions.

\begin{definition}$\mathbb{K}_{C}:=\{G, M, I_{C}\}$ is a conditional context of $\mathbb{K}:=\{G, M, B, Y\}$ relative to conditions $C\subseteq B$ iff $\mathbb{K}_{C}$ satisfies the following property:


For $x_{1}\in G$, $x_{2}\in M$ and $ C\subseteq B$, $x_{1}\times x_{2}\times C\subseteq Y$ iff $x_{1}\times x_{2}\in I_{C}$.

\end{definition}

Benefit from this definition, we can focus on the diadic relations between objects and attributes under certain set of conditions and use the knowledge in diadic formal context for analysing.


\begin{definition}A triadic implication $(R \rightarrow S)_{C}$ holds in a triadic context $\mathbb{K}$, iff each object $g\in G$ having all the attributes in $R$ under all conditions in $C$ has all the attributes in $S$ under all conditions in $C$ as well.
\end{definition}


Inspired by the definition of triadic implication, Lemma 7 specifies how we can define and compute the stem base with respect to the triadic implications.

\begin{lemma}The stem base with respect to triadic implications for $\mathbb{K}:=\{G, M, B, Y\}$ can be defined as the set of stem bases for every conditional context $\mathbb{K}_{C}$ where $C\subseteq B$.
\end{lemma}

\begin{proof}
The following property is first proved to show that there exists a bijective relationship between the triadic implications for $\mathbb{K}:=\{G, M, B, Y\}$ and the implications for every conditional context $\mathbb{K}_{C}$ where $C\subseteq B$.

\emph{``For $R, S\subseteq M$
, $(R \rightarrow S)_{C}$ holds in $\mathbb{K}:=\{G, M, B, Y\}$ iff $R \rightarrow S$ holds in $\mathbb{K}_{C}:=\{G, M, I_{C}\}$."}

1. $R \xrightarrow{C} S$ holds in $\mathbb{K}$ $\Rightarrow$ $R \rightarrow S$ holds in $\mathbb{K}_{C}$.

By Definition 6 and Definition 16, ``$(R \rightarrow S)_{C}$ holds in $\mathbb{K}$'' 
means that ``the set of objects having all attributes of $R$ under all conditions of $C$ also has all attributes in $S$ under all conditions of $C$ in $\mathbb{K}$", without loss of generality, this set of objects is denoted $G_{1}$. In another word, ${G_{1}=(R \times C)}^{(1)}\subseteq {(S \times C)}^{(1)}$, which means $G_{1} \times R \times C\subseteq Y$ and $G_{1} \times S \times C\subseteq Y$. According to Definition 15, ``$G_{1} \times R \times C\subseteq Y$ in $\mathbb{K}$'' implies ``$G_{1} \times R\subseteq I$ in $\mathbb{K}_{C}$''. For the same reason, ``$G_{1} \times S \times C\subseteq Y$ in $\mathbb{K}$'' implies ``$G_{1} \times S\subseteq I$ in $\mathbb{K}_{C}$''. If $G_{1} = R'\subseteq S'$ hold in $\mathbb{K}_{C}$, then follows $R \rightarrow S$ holds in $\mathbb{K}_{C}$.

It is clear that $G_{1} \subseteq S'$ since $G_{1} \times S\subseteq I$. Assume $R' = G_{2}$ and $G_{1}\neq G_{2}$, then $G_{1}\subseteq G_{2}$ and there exists $g\in G_{2}$ and $g\notin G_{1}$. Then follows that $g\times R\subseteq I$. By Definition 15, $g\times R\subseteq I$ implies $g\times R\times C\subseteq Y$, which means $g\in (R\times C)^{(1)}$. However, we already know that $g\notin G_{1}$, conflict occurs. Therefore, $R' = G_{1}$.

Since $G_{1} = R'\subseteq S'$ holds in $\mathbb{K}_{C}$, $R \rightarrow S$ holds in $\mathbb{K}_{C}$ as well.

2.$R \rightarrow S$ holds in $\mathbb{K}_{C}$ $\Rightarrow$ $(R \rightarrow S)_{C}$ holds in $\mathbb{K}$.

By Definition 11, ``$R \rightarrow S$ holds in $\mathbb{K}_{C}$'' means that ``the set of objects having all attributes of R also has all attributes of S in $\mathbb{K}_{C}$", without loss of generality, this set of objects is denoted $G_{1}$. In another word, $G_{1}=R'\subseteq S'$, which means $G_{1} \times R\subseteq I$ and $G_{1} \times S \subseteq I$. According to Definition 15, ``$G_{1} \times R\subseteq I$ in $\mathbb{K}_{C}$'' implies ``$G_{1} \times R \times C\subseteq Y$ in $\mathbb{K}$''. For the same reason, ``$G_{1} \times S\subseteq I$ in $\mathbb{K}_{C}$'' implies ``$G_{1} \times S \times C\subseteq Y$ in $\mathbb{K}$''. If ${G_{1}=(R \times C)}^{(1)}\subseteq {(S \times C)}^{(1)}$, then follows $R \xrightarrow{C} S$ holds in $\mathbb{K}$.

It is clear that $G_{1} \subseteq {(S \times C)}^{(1)}$ since $G_{1} \times S\times C\subseteq Y$. Assume ${(R \times C)}^{(1)} = G_{2}$ and $G_{1}\neq G_{2}$, then $G_{1}\subseteq G_{2}$ and there exists $g\in G_{2}$ and $g\notin G_{1}$. Then follows that $g\times R\times C\subseteq Y$. By Definition 15, $g\times R\times C\subseteq Y$ implies $g\times R\subseteq I$, which means $g\in R'$. However, we already know that $g\notin G_{1}$, conflict occurs. Therefore, ${(R \times C)}^{(1)} = G_{1}$.

Since ${(R \times C)}^{(1)} = G_{1}\subseteq {(S \times C)}^{(1)}$ holds in $\mathbb{K}$, $(R \rightarrow S)_{C}$ holds in $\mathbb{K}$ as well.

Moreover, the stem base should be complete and non-redundant.

Since the stem base of a conditional context can infer all its implications, the set of stem bases for every conditional context is surely complete considering the above mentioned property. As for the proof of its non-redundance, we need to consider two cases. For the case that there does not exist $C_{1}, C_{2}\subseteq B$ such that the conditional contexts $\mathbb{K}_{C_{1}}$ and $\mathbb{K}_{C_{2}}$ are exactly the same, since each conditional context are not replaceable, such stem base for $\mathbb{K}$ is non-redundant. For the rest case that there exists $C_{1}, C_{2}... C_{n}\subseteq B$ such that the conditional contexts $\mathbb{K}_{C_{1}}$, $\mathbb{K}_{C_{2}}$... $\mathbb{K}_{C_{n}}$ are exactly the same, although the stem bases for $\mathbb{K}_{C_{1}}$, $\mathbb{K}_{C_{2}}$... $\mathbb{K}_{C_{n}}$ are the same as well (and seems to be redundant), from the aspect of triadic implications, the stem bases are different sets of triadic implications under $C_{1}, C_{2}... C_{n}$ and can not be inferred from each other since they are from different diadic contexts. Therefore the stem base for $\mathbb{K}$ is non-redundant.\qed\end{proof}

For simplicity, we name the stem base for a conditional context under conditions $C$ the conditional stem base under conditions $C$. Lemma 7 doesn't violate Theorem 1. The stem base introduced in Lemma 7 is complete and non-redundant.

Can we compute the stem base with respect to the triadic implications for $\mathbb{K}$ from the stem bases of $\mathbb{K}_{1}$ and $\mathbb{K}_{2}$ since we use $\mathbb{K}_{1}$ and $\mathbb{K}_{2}$ to construct the triadic context $\mathbb{K}$? Unfortunately no. Table 5 gives us an counter example. Case 1 and case 2 are formal contexts for attributes and meta-attributes, they share the same set of implications $\{\alpha\rightarrow \beta\}$ for meta-attributes, since case 2 is constructed by just exchanging the meta-attributes hold for attributes 2 and 3. However, the corresponding triadic contexts(see Table 6) have different triadic implications. For case 2, the triadic implication is $(1\rightarrow 2)_{\beta}$ while this triadic implication does not satisfy case 1.

\begin{table}[!htb]

\begin{center}
\begin{cxt}%
\cxtName{\tiny{\backslashbox{Objects}{Attributes}}}%
\att{1}%
\att{2}%
\att{3}%
\obj{xx.}{$\color{white}{wwww}$a}
\obj{.xx}{$\color{white}{wwww}$b}
\obj{..x}{$\color{white}{wwww}$c}
\end{cxt}
\end{center}
\begin{center}
\begin{cxt}%
\cxtName{Case 1}%
\att{$\alpha$}%
\att{$\beta$}%
\att{$\gamma$}%
\obj{xx.}{$\color{white}{ff}$1}
\obj{..x}{$\color{white}{ff}$2}
\obj{.x.}{$\color{white}{ff}$3}

\end{cxt}
\begin{cxt}%
\centering
\cxtName{Case 2}%
\att{$\alpha$}%
\att{$\beta$}%
\att{$\gamma$}%
\obj{xx.}{$\color{white}{ff}$1}
\obj{.x.}{$\color{white}{ff}$2}
\obj{..x}{$\color{white}{ff}$3}

\end{cxt}

\end{center}
\caption{Case 1 and Case 2 share the same stem base.}

\end{table}
\begin{table}[!htb]
  \centering
  \begin{tabular}{|c||c|c|c|c|c|c|c|c|c|}

\hline
Case 1
& \multicolumn{2}{r}{$\alpha$} &

& \multicolumn{2}{r}{$\beta$}&&

\multicolumn{2}{r}{$\gamma$}&\\
\hline

\tiny{\backslashbox{Objects}{Attributes}}& 1 & 2 & 3 & 1 & 2 & 3 & 1 & 2 & 3    \\
 \hline
\hline
 a&$\times$&$\color{white}{w} $&$\color{white}{w} $&$\times$&$\color{white}{w} $&&$\color{white}{w} $&$\times$&$\color{white}{w} $\\
 \hline
 b&&&&&&$\times$&&$\times$&\\
 \hline
 c&&&&&&$\times$&&&\\
 \hline

\end{tabular}

  \begin{tabular}{|c||c|c|c|c|c|c|c|c|c|}

\hline
Case 2
& \multicolumn{2}{r}{$\alpha$} &

& \multicolumn{2}{r}{$\beta$}&&

\multicolumn{2}{r}{$\gamma$}&\\
\hline

\tiny{\backslashbox{Objects}{Attributes}}& 1 & 2 & 3 & 1 & 2 & 3 & 1 & 2 & 3    \\
 \hline
\hline
 a&$\times$&$\color{white}{w} $&$\color{white}{w} $&$\times$&$\times$&$\color{white}{w} $&$\color{white}{w} $&$\color{white}{w} $&\\
 \hline
 b&&&&&$\times$&&&&$\times$\\
 \hline
 c&&&&&&&&&$\times$\\
 \hline
\end{tabular}
\caption{Corresponding triadic context for the previous examples.}

\end{table}


However, given the stem base, the intent set of $\mathbb{K}_{1}$ and the set of conditions, we can construct a part of the stem base with respect to the triadic implications for $\mathbb{K}$. Let's have insight of the way we construct such ``partial'' stem base for $\mathbb{K}$.

For a triadic context $\mathbb{K}$, triadic implications are defined as implications under certain conditions(meta-attributes) $C\subseteq B$. In another word, the triadic implications under $C$ are the implications of conditional context relative to $C$. In the case that $\mathbb{K}:=(G, M, B, Y)$ is constructed form $\mathbb{K}_{1}:=(G, M, I)$ and $\mathbb{K}_{2}:=(M, B, J)$, it is clear that each conditional context is constructed from $\mathbb{K}_{1}$ by removing the occurrences of the attributes that don't have all the meta-attributes in $C$. The conditional context with respect to Definition 9 is defined as following.

\begin{definition}Given formal contexts $\mathbb{K}_{1}:=(G, M, I)$, $\mathbb{K}_{2}:=(M, B, J)$ and their corresponding triadic context $\mathbb{K}:=(G, M, B, Y)$, $\mathbb{K}_{\rm{C}}:=(G, M,I/I_{\bar{\text{C}}})$ is the conditional context of $\mathbb{K}$ relative to conditions $C$ iff $I_{\bar{\text{C}}}:=\{(g, m)\in I $ $|$ for $g\in G$ and $m\in M/C^{J}\}$.\end{definition}


\begin{lemma}Given a formal context $\mathbb{C}:=(G, M, I)$, its intent set $INT$, its extent set $EXT$ and its stem base $\mathcal{L}$. After deleting the column of a random attribute $m$ from $M$, we get a new formal context $\mathbb{C}_{*}:=(G, M/m, I_{*})$, where $I_{*}:=I/I_{m}$ and $I_{m}:=\{(g, m)\in I$ $|$ $g\in G\}$. Its intent set, extent set and stem base denoted $INT_{*}$, $EXT_{*}$ and $\mathcal{L_{*}}$ respectively. 

If $A\rightarrow B$ is an implication in $\mathcal{L}$, the following statements hold.

1. If $m\notin B$, $A\rightarrow B$ is in $\mathcal{L_{*}}$.

2. If $m\in B$, $m\notin A$ and $A\neq B/m$, $A\rightarrow B/m$ is in $\mathcal{L_{*}}$.

3. If $m\in A$, find all the implications $A_{1}\rightarrow B_{1}$,$A_{2}\rightarrow B_{2}$,..., $A_{n}\rightarrow B_{n}$ in $\mathcal{L}$, where $m\notin A_{i}$, $m\in B_{i}$ and $B_{i}\subseteq B$ for $1\leq i\leq n$. Every $(A\cup A_{i})/m\rightarrow B/m$ that satisfies $(A\cup A_{i})/m\notin INT_{*}$ is
an implication of $\mathbb{C}_{*}:=(G, M/m, I_{*})$.\end{lemma}

\begin{proof} We now prove each statement.

1. If $m\notin B$, the occurrence of $m$ does not influence the pseudo closed set $A$ and closed set $B$ at all, $A\rightarrow B$ is in $\mathcal{L_{*}}$. 

2. If $m\in B$, $m\notin A$ and $A\neq B/m$, the occurrence of $m$ does not influence the pseudo closed set $A$ at all. On the other hand, $A''=B$ for $\mathbb{C}$ and $m\in B$ implies $A''=B/m$ for $\mathbb{C}_{*}$, as $A\neq B/m$. Then follows that $A\rightarrow B/m$ is in $\mathcal{L_{*}}$.

3. If $m\in A$, find all the implications $A_{1}\rightarrow B_{1}$,$A_{2}\rightarrow B_{2}$,..., $A_{n}\rightarrow B_{n}$ in $\mathcal{L}$, where $m\notin A_{i}$, $m\in B_{i}$ and $B_{i}\subseteq B$ for $1\leq i\leq n$. For $\mathbb{C}$, $B_{i}\subseteq B$ implies $A\cup B_{i}\rightarrow B$. $A_{i}\rightarrow B_{i}$ implies $A\cup A_{i}\rightarrow B$. Moreover, $m\notin A_{i}$ implies $(A\cup A_{i})/m\rightarrow B$. In another word, $((A\cup A_{i})/m)''= B$ holds for $\mathbb{C}$. On the other hand, $m\in B$ implies that $((A\cup A_{i})/m)''= B/m$ holds for $\mathbb{C}_{*}$. While $(A\cup A_{i})/m\notin INT_{*}$, $(A\cup A_{i})/m\rightarrow B/m$ is an implication of $\mathbb{C}_{*}$.\qed 

\end{proof}



Readers may ask if $(A\cup A_{i})/m\rightarrow B/m$ is in $\mathcal{L_{*}}$ as well. A counter example is given in Table 7 and Table 8. Although $\{1,$ $3\}\rightarrow \{1,$ $3,$ $4,$ $5,$ $6\}$ in $\mathcal{L}_{*}$ follows the third statement in Lemma 8, $\{1,$ $3,$ $4\}\rightarrow \{1,$ $3,$ $4,$ $5,$ $6\}$, which can be constructed from $\{1,$ $2,$ $3,$ $4 \}\rightarrow \{1,$ $2,$ $3,$ $4,$ $5,$ $6\}$ in $\mathcal{L}$, is only an implication of $\mathbb{C}_{*}$ and is not included in $\mathcal{L}_{*}$. On the other hand, applying the first two statements of Lemma 8 on $\mathcal{L}$ is not sufficient to compute the whole $\mathcal{L}_{*}$. As the previous example shows, $\{1,$ $3\}\rightarrow \{1,$ $3,$ $4,$ $5,$ $6\}\in \mathcal{L}_{*}$ can not be constructed by using the first two statements of Lemma 8 on $\mathcal{L}$.

\begin{table}[!htb]

\begin{center}
\begin{cxt}%
\cxtName{$\mathbb{C}$}
\att{1}%
\att{2}%
\att{3}%
\att{4}%
\att{5}%
\att{6}%
\obj{xx.x..}{a}
\obj{.xx...}{b}
\obj{xxxxxx}{c}
\obj{....xx}{d}
\obj{.....x}{e}
\obj{x.....}{f}
\end{cxt}
\begin{cxt}%
\cxtName{$\mathbb{C}_{*}$}
\att{1}%
\att{3}%
\att{4}%
\att{5}%
\att{6}%
\obj{x.x..}{a}
\obj{.x...}{b}
\obj{xxxxx}{c}
\obj{...xx}{d}
\obj{....x}{e}
\obj{x....}{f}
\end{cxt}
\caption{$\mathbb{C}_{*}$ is constructed from $\mathbb{C}$ by removing the attribute 2.}
\end{center}
\end{table}
\begin{table}[!htb]
\begin{center}
\begin{tabular}{|c||c|}
\hline
$\mathcal{L}$&$INT$\\
\hline
\hline

$\{5\}\rightarrow \{5,$ $6\}$& $\{6\}$   \\
$\{4\}\rightarrow \{1,$ $2,$ $4\}$&$\{5,$ $6\}$  \\
$\{3\}\rightarrow \{2,$ $3\}$& $\{2\}$\\
$\{2,$ $6\}\rightarrow \{1,$ $2,$ $3,$ $4,$ $5,$ $6\}$&$\{2,$ $3\}$ \\
$\{1,$ $6\}\rightarrow \{1,$ $2,$ $3,$ $4,$ $5,$ $6\}$&$\{1\}$ \\
$\{1,$ $2\}\rightarrow \{1,$ $2,$ $4\}$&$\{1,$ $2,$ $4\}$ \\
$\{1,$ $2,$ $3,$ $4 \}\rightarrow \{1,$ $2,$ $3,$ $4,$ $5,$ $6\}$&$\{1,$ $2,$ $3,$ $4,$ $5,$ $6\}$    \\

 \hline
 \end{tabular}

\begin{tabular}{|c||c|}
\hline
$\mathcal{L}_{*}$&$INT_{*}$\\
\hline
\hline

$\{5\}\rightarrow \{5,$ $6\}$& $\{6\}$   \\
$\{4\}\rightarrow \{1,$ $4\}$&$\{5,$ $6\}$  \\
$\{3,$ $6\}\rightarrow \{1,$ $3,$ $4,$ $5,$ $6\}$& $\{3\}$\\
$\{1,$ $6\}\rightarrow \{1,$ $3,$ $4,$ $5,$ $6\}$&$\{1\}$ \\
$\{1,$ $3\}\rightarrow \{1,$ $3,$ $4,$ $5,$ $6\}$&$\{1,$ $4\}$ \\
&$\{1,$ $3,$ $4,$ $5,$ $6\}$    \\
 \hline
 \end{tabular}

\end{center}
\caption{The stem bases and intent sets for $\mathbb{C}$ and $\mathbb{C}_{*}$.}

\end{table}

Lemma 8 specifies the case that one attribute is removed for a formal context, it can be generalized to the case that some attributes $M_{\rm{remove}}\subset M$ are removed for a formal context. As we focus on constructing the stem base, the third statement of Lemma 8 is not concerned any more.

\begin{lemma}Given a formal context $\mathbb{C}:=(G, M, I)$, its intent set $INT$, its extent set $EXT$ and its stem base $\mathcal{L}$. After deleting the columns of random attributes $M_{\rm{remove}}$ from $M$, we get a new formal context $\mathbb{C}_{*}:=(G, M/M_{\rm{remove}}, I_{*})$, where $I_{*}:=I/I_{M_{\rm{remove}}}$ and $I_{M_{\rm{remove}}}:=\{(g, M_{\rm{remove}})\in I$ $|$ $g\in G\}$. Its intent set, extent set and stem base denoted $INT_{*}$, $EXT_{*}$ and $\mathcal{L_{*}}$ respectively. 

If $A\rightarrow B$ is an implication in $\mathcal{L}$, the following statements hold.

1. If $M_{\rm{remove}}\cap B = \emptyset$, $A\rightarrow B$ is in $\mathcal{L_{*}}$.

2. If $M_{\rm{remove}}\cap B \neq \emptyset$, $M_{\rm{remove}}\cap A=\emptyset$ and $A\neq B/M_{\rm{remove}}$, $A\rightarrow B/M_{\rm{remove}}$ is in $\mathcal{L_{*}}$.
\end{lemma}

\begin{proof} The proof is similar to the proof of Lemma 8.

1. If $M_{\rm{remove}}\cap B = \emptyset$, the occurrence of $M_{\rm{remove}}$ does not influence the pseudo closed set $A$ and closed set $B$ at all, $A\rightarrow B$ is in $\mathcal{L_{*}}$.

2. If $M_{\rm{remove}}\cap B \neq \emptyset$, $M_{\rm{remove}}\cap A=\emptyset$ and $A\neq B/M_{\rm{remove}}$, the occurrence of $M_{\rm{remove}}$ does not influence the pseudo closed set $A$ at all. On the other hand, $A''=B$ for $\mathbb{C}$ and $M_{\rm{remove}}\cap B \neq \emptyset$ implies $A''=B/M_{\rm{remove}}$ for $\mathbb{C}_{*}$, as $A\neq B/M_{\rm{remove}}$. Then follows that $A\rightarrow B/M_{\rm{remove}}$ is in $\mathcal{L_{*}}$.\qed

\end{proof}

Algorithm 3 shows that how we can construct part of $\mathcal{L}_{*}$, denoted $\mathcal{L}_{*\rm{partial}}$, from $\mathcal{L}$ by using Lemma 9. This algorithm makes a contribution to simplifying the computation of $\mathcal{L}_{*}$.

\begin{algorithm}
\caption{Generating $\mathcal{L}_{*\rm{partial}}$ from $\mathcal{L}$}
 Input: $\mathbb{C}:=(G,M,I), M_\text{remove}$ and $\mathcal{L}$.

 Output: $\mathcal{L}_{*\rm{partial}}$ and $\mathbb{C}_{*}$.

 For every implication $A\rightarrow B\in \mathcal{L}$, do the following:
  \begin{algorithmic}[1]

  \If {$M_{\rm{remove}}\cap B = \emptyset$}
  \State $\mathcal{L}_{*\rm{partial}}\leftarrow \mathcal{L}_{*\rm{partial}}\cup A\rightarrow B$
  \Else
  \If{$M_{\rm{remove}}\cap B \neq \emptyset$, $M_{\rm{remove}}\cap A=\emptyset$ and $A\neq B/m$}
  \State $\mathcal{L}_{*\rm{partial}}\leftarrow \mathcal{L}_{*\rm{partial}}\cup A\rightarrow B/M_{\rm{remove}}$
  \EndIf
   \EndIf
\State $\mathbb{C}_{*}\leftarrow (G, M/M_\text{remove}, I/I_{M_\text{remove}})$

  \end{algorithmic}
\end{algorithm}

For the variant that we remove the occurrences of $M_\text{remove}$ rather then delete the columns of $M_\text{remove}$, the attributes in $M_\text{remove}$ also makes a contribution to the stem base of $\mathbb{C}_{\star}$.

\begin{lemma}Given a formal context $\mathbb{C}:=(G, M, I)$, its intent set $INT$, its extent set $EXT$ and its stem base $\mathcal{L}$. After removing the occurrences of random attributes $M_{\rm{remove}}$ from $M$, we get a new formal context $\mathbb{C}_{\star}:=(G, M, I_{\star})$, where $I_{\star}:=I/I_{M_{\rm{remove}}}$ and $I_{M_{\rm{remove}}}:=\{(g, M_{\rm{remove}})\in I$ $|$ $g\in G\}$. Its intent set, extent set and stem base denoted $INT_{\star}$, $EXT_{\star}$ and $\mathcal{L_{\star}}$ respectively.

If $A\rightarrow B$ is an implication in $\mathcal{L}$, the following statements hold.

1. If $M_{\rm{remove}}\cap B = \emptyset$, $A\rightarrow B$ is in $\mathcal{L_{\star}}$.

2. If $M_{\rm{remove}}\cap B \neq \emptyset$, $M_{\rm{remove}}\cap A=\emptyset$ and $A\neq B/M_{\rm{remove}}$, $A\rightarrow B/M_{\rm{remove}}$ is in $\mathcal{L_{\star}}$.

3. For every $m\in M_{\rm{remove}}$, $\{m\}\rightarrow M$ is in $\mathcal{L_{\star}}$.
\end{lemma}
\begin{proof} Since $M_\text{remove}$ relate to no object, the first two statements holds directly according to Lemma 9. For every $m\in M_{\rm{remove}}$, the implication $\{m\}\rightarrow M$ holds and $m$ satisfies Definition 14, thus $\{m\}\rightarrow M$ is in $\mathcal{L_{\star}}$.
\qed\end{proof}

\begin{algorithm}
\caption{Generating $\mathcal{L}_{\star \rm{partial}}$ from $\mathcal{L}$}
 Input: $\mathbb{C}$, $M_\text{remove}:=\{m_{1}, m_{2},..., m_{n}\}$ and $\mathcal{L}$.

 Output: $\mathcal{L}_{\star \rm{partial}}$, $\mathbb{C}_{\star}$

 For every implication $A\rightarrow B\in \mathcal{L}$, do the following:
  \begin{algorithmic}[1]

  \If {$M_{\rm{remove}}\cap B = \emptyset$}
  \State $\mathcal{L}_{*\rm{partial}}\leftarrow \mathcal{L}_{*\rm{partial}}\cup A\rightarrow B$
  \Else
  \If{$M_{\rm{remove}}\cap B \neq \emptyset$, $M_{\rm{remove}}\cap A=\emptyset$ and $A\neq B/m$}
  \State $\mathcal{L}_{*\rm{partial}}\leftarrow \mathcal{L}_{*\rm{partial}}\cup A\rightarrow B/M_{\rm{remove}}$
  \EndIf
   \EndIf
   \State $i\leftarrow 0$
\Repeat
\State $i++$
\State $\mathcal{L}_{*\rm{partial}}\leftarrow \mathcal{L}_{*\rm{partial}}\cup (m_{i}\rightarrow M)$

\Until $i=n$
\State $\mathbb{C}_{\star}\leftarrow (G, M, I/I_{M_\text{remove}})$
  \end{algorithmic}
\end{algorithm}

Algorithm 4 shows that how we can compute  $\mathcal{L}_{\star \rm{partial}}$. The procedure is similar to Algorithm 3. While maintaining the operations in Algorithm 3, from step 7 to step 10, Algorithm 4 specifies the implications that satisfies the third statement in Lemma 10.

In our case, the conditional contexts are constructed by removing the occurrences of the attributes who have no relevant meta-attributes. We can find out all the formal concepts for certain conditional context by using the above mentioned method. This makes a contribution to computing the triadic implications while using the Next Closure algorithm.
\subsection{Modified Next Closure Algorithm for Computing Triadic Implications}

We find out the way to construct all the new formal concepts after removing any attribute $m$|simply remove all the occurrences of $m$. On the other hand, part of the conditional stem bases can be constructed. These conclusions will help us modify the Next Closure algorithm to meet our need.
\subsubsection{Next Closure Algorithm}

The algorithms is developed in 1984 by Bernhard Ganter.\cite{ganter2012book} It is an algorithm used to find out the intents and stem base in the lectic order. It can be rephrased without reference to formal concept analysis since that it essentially relies on a single property of concept
lattices, that the set of concept intents is closed under intersections. The
technique can be formulated for closure systems.

\begin{definition} Let $M=\{1,...,n\}$. We say that $A\subseteq M$ is lectically smaller than $B\subseteq M$, if $B\neq A$ and the smallest element in which A and B differ belongs to B. Moreover, ``$<_{i}$'' and ``$+$'' are defined as shown below.\end{definition}

\begin{center}
$A<_{i}B$$\Leftrightarrow i\in B/A\wedge A\cap\{1,2,...,i-1\}=B\cap\{1,2,...,i-1\}$.

$A+i=(A\cap\{1,2,...,i-1\})\cup i.$
\end{center}
\begin{table}[!htb]
\begin{diagram}{0}{160}
\Node{0}{60}{10}
\Node{1}{40}{40}
\Node{2}{60}{40}
\Node{3}{100}{40}
\Node{4}{40}{70}
\Node{5}{80}{70}
\Node{6}{100}{70}
\Node{7}{80}{100}
\Edge{0}{1}
\Edge{0}{2}
\Edge{0}{3}
\Edge{4}{2}
\Edge{3}{5}
\Edge{3}{6}
\Edge{7}{6}
\leftObjbox{0}{2}{2}{$\phi$}
\rightAttbox{1}{2}{2}{5}
\rightAttbox{2}{2}{2}{4}
\rightAttbox{3}{2}{2}{3}
\rightAttbox{4}{2}{2}{45}
\rightAttbox{5}{2}{2}{35}
\rightAttbox{6}{2}{2}{34}
\rightAttbox{7}{2}{2}{345}
\end{diagram}
\begin{diagram}{0}{150}
\Node{0}{136.5}{40}
\Node{8}{56.5}{10}
\Node{1}{116.50}{70}
\Node{2}{136.50}{70}
\Node{3}{176.5}{70}
\Node{4}{116.5}{100}
\Node{5}{156.5}{100}
\Node{6}{176.5}{100}
\Node{7}{156.5}{130}
\Edge{0}{1}
\Edge{0}{8}
\Edge{0}{2}
\Edge{0}{3}
\Edge{4}{2}
\Edge{3}{5}
\Edge{3}{6}
\Edge{7}{6}
\rightAttbox{0}{2}{-4}{2}
\rightAttbox{1}{2}{2}{25}
\rightAttbox{2}{2}{2}{24}
\rightAttbox{3}{2}{2}{23}
\rightAttbox{4}{2}{2}{245}
\rightAttbox{5}{2}{2}{235}
\rightAttbox{6}{2}{2}{234}
\rightAttbox{7}{2}{2}{2345}
\end{diagram}
\begin{diagram}{0}{150}
\Node{0}{213}{40}
\Node{8}{53}{10}
\Node{1}{193}{70}
\Node{2}{213}{70}
\Node{3}{253}{70}
\Node{4}{193}{100}
\Node{5}{233}{100}
\Node{6}{253}{100}
\Node{7}{233}{130}
\Edge{0}{1}
\Edge{0}{8}
\Edge{0}{2}
\Edge{0}{3}
\Edge{4}{2}
\Edge{3}{5}
\Edge{3}{6}
\Edge{7}{6}
\rightAttbox{0}{2}{-4}{1}
\rightAttbox{1}{2}{2}{15}
\rightAttbox{2}{2}{2}{14}
\rightAttbox{3}{2}{2}{13}
\rightAttbox{4}{2}{2}{145}
\rightAttbox{5}{2}{2}{135}
\rightAttbox{6}{2}{2}{134}
\rightAttbox{7}{2}{2}{1345}
\end{diagram}
\begin{diagram}{0}{150}
\Node{0}{289.5}{70}
\Node{8}{209.5}{40}
\Node{1}{269.5}{100}
\Node{2}{289.5}{100}
\Node{3}{329.5}{100}
\Node{4}{269.5}{130}
\Node{5}{309.5}{130}
\Node{6}{329.5}{130}
\Node{7}{309.5}{160}
\Edge{0}{1}
\Edge{0}{8}
\Edge{0}{2}
\Edge{0}{3}
\Edge{4}{2}
\Edge{3}{5}
\Edge{3}{6}
\Edge{7}{6}
\rightAttbox{0}{2}{-4}{12}
\rightAttbox{1}{2}{2}{125}
\rightAttbox{2}{2}{2}{124}
\rightAttbox{3}{2}{2}{123}
\rightAttbox{4}{2}{2}{1245}
\rightAttbox{5}{2}{2}{1235}
\rightAttbox{6}{2}{2}{1234}
\rightAttbox{7}{2}{2}{12345}
\end{diagram}
\begin{center}
  Fig 3: The set $\{1, 2, 3, 4, 5\}$ represented by lectic ordered tree.
\end{center}
\end{table}
Figure 3 shows a tree follows the lectic order. Each node uniquely stands for a subset of $S=\{1, 2, 3, 4, 5\}$ and the union of the subsets is $2^{|S|}$. Such a tree holds the following properties:

1. Each node is lectically smaller than its child nodes.

2. For two node in the same depth $a$ and $b$, if $a$ is on the left of $b$, $a$ is lectically smaller than $b$.

Algorithm 3 shows the original Next Closure algorithm to compute all concept intents and the stem base|it finds out all the intents and pseudo intents in lectic order. Interested readers can find more details in \cite{ganter2012book,ganter2002formal}.




\begin{algorithm}[!htb]
\caption{ Next Closure}

Input: $\mathbb{C}$.

Output: $INT$, $\mathcal{L}$.



{$A$ is a subset of $M$.}
  \begin{algorithmic}[1]
  \State $A \leftarrow\emptyset$, $\mathcal{L} \leftarrow\emptyset$, $INT \leftarrow\emptyset$
  \State $Temp\leftarrow M/A$, $ i\leftarrow max(Temp)$
  \Repeat
  \If{$A<_{i} \mathcal{L}(A+i)$}
  \State{$A\leftarrow \mathcal{L}(A+i)$}

  \If{$A= A''$}
  \State $INT\leftarrow INT\cup A$
  \Else
  \State{$\mathcal{L}\leftarrow \mathcal{L}\cup (A\rightarrow A'')$ }
   \EndIf


\State $Temp\leftarrow M/A$, $ i\leftarrow max(Temp)$
   \Else
\State $Temp\leftarrow Temp/i$, $i\leftarrow max(Temp)$

   \EndIf
   \Until {{$A =M$}}
   \State finish
  \end{algorithmic}
\end{algorithm}



\subsubsection{Modified Next Closure Algorithm}

We can modify the \emph{Next Closure} algorithm to make it more efficient in our case for computing triadic implications. The idea is to skip the redundant computations when compute the intents and the conditional stem bases, as the intents and part of stem bases for conditional contexts can be easily computed via Lemma 4 and lemma 10. Given formal contexts $\mathbb{K}_{1}:=(G, M, I)$ and $\mathbb{K}_{2}:=(M, B, J)$ and their corresponding triadic context $\mathbb{K}:=(G, M, B, Y)$. We first find out all the intents and the stem base of $\mathbb{K}_{1}$. Apply Algorithm 1 and Algorithm 3 to each conditional context to construct its intents and part of its stem base. As we have already constructed the corresponding set of intents and part of stem bases, we can use the following modified Next Closure algorithm.





\begin{algorithm}[!htb]
\caption{Modified Next Closure}

Input: $\mathbb{C}$, $INT$, $\mathcal{L}_{\rm{partial}}$.

Output: $\mathcal{L}$.


{$A$ is a subset of $M$.}
   \begin{algorithmic}[1]
  \State $A \leftarrow\emptyset$
  \State $Temp\leftarrow M/A$, $ i\leftarrow max(Temp)$
  \Repeat
  \If{$A<_{i} \mathcal{L}(A+i)$}

  \State{$A\leftarrow \mathcal{L}(A+i)$}
  \If{$(A\rightarrow B)\in \mathcal{L}_{\rm{partial}}$}
  \State{$\mathcal{L}\leftarrow \mathcal{L}\cup (A\rightarrow B)$, $\mathcal{L}_{\rm{partial}}\leftarrow \mathcal{L}_{\rm{partial}}/ (A\rightarrow B)$}
  \Else\If{$A\notin INT$}

  \State{$\mathcal{L}\leftarrow \mathcal{L}\cup (A\rightarrow A'')$}
   \EndIf
   \EndIf


\State $Temp\leftarrow M/A$, $ i\leftarrow max(Temp)$
   \Else
\State $Temp\leftarrow Temp/i$, $i\leftarrow max(Temp)$
   \EndIf
   \Until {$A =M$}
   \State finish
  \end{algorithmic}
\end{algorithm}





In step 6, we first check if there exists $(A\rightarrow B)\in \mathcal{L}_{\rm{partial}}$ for the current $A$ that we are dealing with. If so, $A\rightarrow B$ is added to $\mathcal{L}$ and it is removed from $\mathcal{L}_{\rm{partial}}$. In step 9, we only need to check if $A$ is an intent and save the time that used to compute $A''$ in the original algorithm. Moreover, in our case, the intents under different combinations of condition can be constructed recursively. Consider $C \subseteq B, c_{n}\in B/C, n \leq |B|$, the intents under conditions $C\cup c_{n}$ can be constructed from the ones under conditions $C$ via Algorithm 1.
\subsubsection{Efficient Planning for Computing the Triadic Implications}

For a triadic context $\mathbb{K}:=(G, M, B, Y)$, there are $2^{|B|}$ combinations of conditions, which means we need to check $2^{|B|}$ conditional contexts to find out all the conditional stem bases. When the number of meta-attributes is very large, the number of conditional contexts will increase exponentially. However, many of the combinations share no common relations between objects and attributes, which means no implication exists under such combinations of conditions. The modified Next Closure algorithm needs to be further improved to get rid of the uninteresting combinations.

For a triadic context $\mathbb{K}:=\{G, M, B, Y\}$ defined in Definition 17, $C_{1}\subseteq B$ and $C_{2}\subseteq B$, if $C_{2}$ has the same conditional context with $C_{1}$\footnote{The chance is very high since, in our case, the conditional context under conditions $C$ can be constructed from $\mathbb{K}_{1}$ by removing the attributes that don't have $C$ as their conditions} and we already know the concepts and implications under $C_{1}$, we should not spend time on computing the concepts and implications for the same context. As all the conditional contexts of $\mathbb{K}$ share the same set of objects, $C_{1}^J = C_{2}^J$ implies that $C_{1}$ and $C_{2}$ share the same conditional context. In another word, if $C_{1}$ and $C_{2}$ hold for the same attributes, their conditional contexts are the same. Considering the above mentioned situations, we introduce Algorithm 7.





\begin{algorithm}[!htb]
\caption{Procedure for Computing the Conditional Stem Bases}

  Input: $\mathbb{K}_{1}$, $\mathbb{K}_{2}$ and $\mathbb{K}$.

  Output: $List$, $List'$ and $List_\mathcal{L}$.

  \begin{algorithmic}[1]
  \State $C\leftarrow \emptyset$

  \State Algorithm 5($\mathbb{K}_{1}$)
  \State $A_{0}\leftarrow M$, $\mathcal{L}_{0}\leftarrow\mathcal{L}_{\mathbb{K}_{1}}$, $INT\leftarrow INT_{\mathbb{K}_{1}}$

\State $\# List\leftarrow 0$
\Repeat
\State $ i\leftarrow max(B/C)$

\State $C\leftarrow C+i$

\State $j\leftarrow 0$
\State $Skip\leftarrow$ false
\Repeat

\If {$A_{j}={C}^{J}$}
\State $\mathcal{L}_{\mathbb{K}_C}\leftarrow \mathcal{L}_{j}$
\State $Skip\leftarrow$ true
\State $j\leftarrow\# List$
\EndIf
\State $j++$

\Until $j>\# List$
\If {$Skip=$ false}
\State Algorithm 2($\mathbb{K}_{1}$, $INT$, $M_\text{remove}$)
\State Algorithm 4($\mathbb{K}_{1}$, $\mathcal{L}_{0}$, $M_\text{remove}$)
\State Algorithm 6($\mathbb{K}_{C}$, $INT_{\mathbb{K}_C}$, $\mathcal{L}_{\mathbb{K}_C\text{partial}}$)\EndIf
\State $C_{j}\leftarrow C$, $A_{j}\leftarrow C^J$, $\mathcal{L}_{j}\leftarrow\mathcal{L}_{\mathbb{K}_\text{{C}}}$
\State $\# List++$


\Until $C=B$
\State finish
  \end{algorithmic}
\end{algorithm}

Algorithm 7 shows that how we can find out all the conditional stem bases for every subset of $B$ while ignoring all the `redundant' sets of conditions. In Algorithm 7, $List$ stores all the non-redundant sets of conditions ${C}_{1},{C}_{2},..., {C}_{n}$; $List'$ stores the sets $A_{1}={C}_{1}^J,A_{2}={C}_{2}^J,..., A_{n}={C}_{n}^J$ and $List_\mathcal{L}$ stores the stem bases $\mathcal{L}_{1}=\mathcal{L}_{\mathbb{K}_{C_{1}}},\mathcal{L}_{2}=\mathcal{L}_{\mathbb{K}_{C_{2}}},..., \mathcal{L}_{n}=\mathcal{L}_{\mathbb{K}_{C_{n}}}$. $\mathcal{L}_{\mathbb{K}_C}$ stands for the stem base of $\mathbb{K}_C$. All of the three are initialized as empty. The algorithm starts with $\mathbb{K}_{1}$ since the conditional contexts can be constructed from it. The algorithm applies Next Closure algorithm on $\mathbb{K}_{1}$. $M$, $\mathcal{L}_{\mathbb{K}_{C}}$ and $INT_{\mathbb{K}_{C}}$ are denoted to $A_{0}$, $INT$ and $\mathcal{L}_{0}$ respectively. Then the algorithm starts to check all the subsets $C$ of $B$ in lectic order. For each subset $C$, if it has the same conditional context with some conditional context that we have already checked, we should not compute the implications again. To be more precise, if $C^{J}=A_{j}$ for some $A_{j}\subseteq List'$, $\mathcal{L}_{j}$ is assigned to $\mathcal{L}_{\mathbb{K}_C}$ and $C$, $C^J$ and $\mathcal{L}_{\mathbb{K}_C}$ are added to $List$, $List'$ and $List_\mathcal{L}$ respectively. Otherwise, the algorithm uses Modified Next Closure algorithm to compute the triadic implications and $C$, $C^J$ and $\mathcal{L}_{\mathbb{K}_C}$ are added to $List$, $List'$ and $List_\mathcal{L}$ respectively.

\begin{algorithm}[!htb]
\caption{Procedure for Computing the Conditional Stem Bases}

  Input: $\mathbb{K}$.

  Output: $List$, $List'$ and $List_\mathcal{L}$.

  \begin{algorithmic}[1]
  \State $C\leftarrow \emptyset$, $j = 0$


\Repeat

  \State $i \leftarrow max(B/C)$

\State $C\leftarrow C+i$

\State $j++$
\State Algorithm 5($\mathbb{K}_C$)
\State $C_{j}\leftarrow C$, $A_{j}\leftarrow C^J$, $\mathcal{L}_{j}\leftarrow\mathcal{L}_{\mathbb{K}_\text{{C}}}$




\Until $C=B$
\State finish
  \end{algorithmic}
\end{algorithm}

 Since Algorithm 7 can not be applied to the triadic context other than the cases that are defined by Definition 17, we sketch another more generalized algorithm to compute the triadic implications for any triadic context. However, Algorithm 8 doesn't consider the case that different sets of conditions may share the same conditional context, since the chance is not that big for the triadic contexts other
than our case.





\subsection{Conditional Attribute Implications in Triadic Formal Contexts}
Ganter and Obiedkov have discussed another type of implication for triadic contexts, which is called conditional implication.

\begin{definition}A conditional (attribute) implication $R \xrightarrow{C} S$ holds in a triadic context $\mathbb{K}$, iff for each condition $c\in C$ it holds that if an object $g\in G$ has all the attributes in $R$, it also has all the attributes in $S$.\end{definition}

$R \xrightarrow{C} S$ is read as ``$R$ implies $S$ under all conditions from $C$''.\footnote{In our case, may be read as ``$R$ implies $S$ under all meta-attributes from $C$''.} It is easy to see that $R \xrightarrow{C} S\Leftrightarrow (R \rightarrow S)_c$ for all $c\in C$ from definition 19 and definition 16.

\begin{lemma}The stem base with respect to conditional implications for $\mathbb{K}:=\{G, M, B, Y\}$ can be defined as the set of stem bases for every conditional context $\mathbb{K}_{c}$ where $c\in B$.
\end{lemma}

\begin{proof}To prove that the set of stem bases for every conditional context is the stem base relevant to conditional implication, we need to prove that it is complete and non-redundant.

1. From the stem base we can infer all the conditional implications.(completeness)

As it is the set of stem bases  of conditional contexts $K_{c_i}$  for $1\leq i\leq |C|$, we can find all the implications for each $K_{c_i}$, which means that we can find out all the conditional implications under $c_i$. By the Definition 19, if an implication, for example, $R\rightarrow S$ holds for every condition $c_i, c_{i+1},..., c_j$ ,$1\leq i<j\leq |C|$, $R \xrightarrow{c_i, c_{i+1},..., c_j} S$ is also a conditional implication. Knowing all the conditional implications under each $c_i$, a good way to find out all the conditional implications of $\mathbb{K}$ is to define another formal context $\mathbb{C}\text{imp}(\mathbb{K}):=(\text{Imp}, B, I)$, where $\text{Imp}:=\{R\rightarrow S$ $|$ $R\rightarrow S \text{ is an implication of } K_{c_i} \text{ for } 1\leq i\leq |C|\}$, $I:=\{(R\rightarrow S) \times C_i $ $|$ $R\xrightarrow{C_{i}} S$ is a conditional implication of $\mathbb{K}\}$. We can see that for every possible conditional implication $R\xrightarrow{C_{i}} S$ of $\mathbb{K}$, $(R\rightarrow S) \times C_i \subseteq I$. Therefore, this stem base relevant to conditional implication is complete.

2. It is non-redundant.

As it is the set of stem bases, each stem base is non-redundant and each $K_{c_i}$ is uniquely defined and constructed, the set of stem bases is non-redundant as well.
\qed \end{proof}

Benefit from Lemma 11, we can compute the stem base with respect to conditional implications for $\mathbb{K}:=\{G,M,B,Y\}$ by computing the stem base for every conditional context $\mathbb{K}_{c}$ where $c\subseteq B$. The procedure is similar to the way we compute triadic implications. However, what we need to concern now are every $c\in B$ rather than every $C\subseteq B$, which means that we don't even need the lectic order to traverse the subsets.
\begin{algorithm}[!htb]
\caption{Efficient Planning for Computing the Conditional Stem Bases}

  Input: $\mathbb{K}_{1}$, $\mathbb{K}_{2}$ and $\mathbb{K}$.

  Output: $List$, $List'$ and $List_\mathcal{L}$.

  \begin{algorithmic}[1]
  \State $c\leftarrow \emptyset$, $B_\text{left}\leftarrow B$

  \State Algorithm 5($\mathbb{K}_{1}$)
  \State $A_{0}\leftarrow M$, $\mathcal{L}_{0}\leftarrow\mathcal{L}_{\mathbb{K}_{1}}$, $INT\leftarrow INT_{\mathbb{K}_{1}}$

\State $\# List\leftarrow 0$
\Repeat

\State $ c\leftarrow max(B_\text{left})$

\State $j\leftarrow 0$
\State $Skip\leftarrow$ false
\Repeat

\If {$A_{j}={c}^{J}$}
\State $\mathcal{L}_{\mathbb{K}_{c}}\leftarrow \mathcal{L}_{j}$
\State $Skip\leftarrow$ true
\State $j\leftarrow\# List$
\EndIf
\State $j++$

\Until $j>\# List$
\If {$Skip=$ false}
\State Algorithm 2($\mathbb{K}_{1}$, $INT$, $M_\text{remove}$)
\State Algorithm 4($\mathbb{K}_{1}$, $\mathcal{L}_{0}$, $M_\text{remove}$)
\State Algorithm 6($\mathbb{K}_{c}$, $INT_{\mathbb{K}_{c}}$, $\mathcal{L}_{\mathbb{K}_{c}\text{partial}}$)\EndIf
\State $c_{j}\leftarrow c$, $A_{j}\leftarrow c^J$, $\mathcal{L}_{j}\leftarrow\mathcal{L}_{\mathbb{K}_{c}}$
\State $\# List++$

 \State $ B_\text{left}\leftarrow B_\text{left}/c$
\Until $B_\text{left}=\emptyset$
\State finish
  \end{algorithmic}
\end{algorithm}
Algorithm 9 shows that how we can find out all the conditional stem bases for every $b\in B$ while ignoring all the `redundant' sets of conditions. In Algorithm 9, $List$ stores every condition ${c}_{1},{c}_{2},..., {c}_{n}$ of $B$; $List'$ stores the sets $A_{1}={c}_{1}^J,A_{2}={c}_{2}^J,..., A_{n}={c}_{n}^J$ and $List_\mathcal{L}$ stores the stem bases $\mathcal{L}_{1}=\mathcal{L}_{\mathbb{K}_{c_{1}}},\mathcal{L}_{2}=\mathcal{L}_{\mathbb{K}_{c_{2}}},..., \mathcal{L}_{n}=\mathcal{L}_{\mathbb{K}_{c_{n}}}$. $\mathcal{L}_{\mathbb{K}_{c}}$ stands for the stem base of $\mathbb{K}_c$. All of the three are initialized as empty. The algorithm is very similar to Algorithm 7. Instead of checking all the subset of $B$ in lectic order, Algorithm 9 checks every condition in descending order. $B_\text{left}$ stores the conditions that are not checked yet.

\begin{algorithm}[!htb]
\caption{Procedure for Computing the Conditional Stem Bases}

  Input: $\mathbb{K}$.

  Output: $List$, $List'$ and $List_\mathcal{L}$.

  \begin{algorithmic}[1]
  \State $c\leftarrow \emptyset$, $j = 0$, $B_\text{left}\leftarrow B$


\Repeat

  \State $ c\leftarrow max(B_\text{left})$


\State $j++$
\State Algorithm 5($\mathbb{K}_c$)
\State $C_{j}\leftarrow C$, $A_{j}\leftarrow C^J$, $\mathcal{L}_{j}\leftarrow\mathcal{L}_{\mathbb{K}_c}$

\State $ B_\text{left}\leftarrow B_\text{left}/c$
\Until $B_\text{left}=\emptyset$
\State finish
  \end{algorithmic}
\end{algorithm}

 Since Lemma 11 can not be applied to the triadic context other than the cases that are defined by Definition 17, we sketch another more generalized algorithm to compute the conditional implications for any triadic context. We simply check every $b\in B$ in descending order and compute the relevant conditional stem base by using the Next Closure algorithm.
\section{Conclusion and Future Tasks}

The ternary relationship among {objects, attributes} and {meta-attributes} can be presented in a triadic context and visualized in a triadic diagram. By Lemma 5 and Lemma 6, the (diadic) formal concepts are reserved and can be retrieved in the triadic diagram. The conditional stem bases can be computed by using the modified {Next Closure} algorithm in our case. Moreover, we present the method to ignore the uninteresting conditional contexts while computing the stem base with respect to the triadic implications and conditional implications. On the other hand, Lemma 10 is not sufficient to compute the complete stem bases. Extending Lemma 10 to compute the complete stem base will be a future task to discuss.


\bibliographystyle{splncs}
\bibliography{meta1}

\begin{thebibliography}{1}

\bibitem{wille2009restructuring}
Wille, R.:
\newblock Restructuring lattice theory: An approach based on hierarchies of concepts.
\newblock In: Proceedings of the 7th International Conference on Formal Concept Analysis. ICFCA 2009, Berlin, Heidelberg, Springer-Verlag (2009)  314--339

\bibitem{lehmann1995triadic}
Lehmann, F., Wille, R.:
\newblock A triadic approach to formal concept analysis.
\newblock In: Proceedings of the Third International Conference on Conceptual Structures: Applications, Implementation and Theory, Berlin-Heidelberg-New York, Springer-Verlag (1995)  32--43

\bibitem{ganter2012book}
Ganter, B., Wille, R.:
\newblock Formal concept analysis: mathematical foundations.
\newblock Springer Science \& Business Media (1999)

\bibitem{ganter2002formal}
Ganter, B., Stumme, G.:
\newblock Formal concept analysis: Methods and applications in computer science.
\newblock (2002)

\bibitem{Biedermann:1997:TDR:645490.657192}
Biedermann, K.:
\newblock How triadic diagrams represent conceptual structures.
\newblock In: Proceedings of the Fifth International Conference on Conceptual Structures: Fulfilling Peirce's Dream. ICCS 1997, London, UK, Springer-Verlag (1997)  304--317

\bibitem{ganter2004implications}
Ganter, B., Obiedkov, S.:
\newblock Implications in triadic formal contexts.
\newblock In: Proceedings of the 12th International Conference on Conceptual Structures: Conceptual Structures at Work, ICCS 2004, Huntsville, AL, USA, July 19-23, Springer (2004)  186--195

\end{thebibliography}

\end{document}